\documentclass{amsart}

\usepackage{amsmath, amsfonts, amsthm, amssymb}
\usepackage{graphicx}
\usepackage{subcaption}
\usepackage{mathrsfs}
\usepackage{bm}
\usepackage{tikz}
\usepackage{tikz-3dplot}
\usepackage{tikz-qtree}
\usepackage{forest}

\usepackage{libertine}
\usepackage[libertine]{newtxmath}

\usepackage{natbib}

\usepackage[
    colorlinks=true, 
    linkcolor=red,
    citecolor=blue
]{hyperref}

\usepackage{orcidlink}

\newtheorem{theorem}{Theorem}[section]
\newtheorem{proposition}[theorem]{Proposition}

\newtheorem{lemma}[theorem]{Lemma}
\theoremstyle{definition}
\newtheorem{definition}[theorem]{Definition}
\newtheorem{example}[theorem]{Example}
\newtheorem{remark}[theorem]{Remark}
 
\usepackage{bussproofs}

\usepackage{enumitem}

\setlist[itemize]{
    labelindent=0em, 
    leftmargin=3\parindent, 
    labelsep=0pt, 
}

\setlist[enumerate]{
    labelindent=0em, 
    leftmargin=3\parindent, 
    labelsep=0pt, 
    label=\makebox[2em]{\centering(\arabic*)}, 
}

\setlist[enumerate, 2, 3, 4]{
    labelindent=0em, 
    leftmargin=\parindent, 
    labelsep=0pt, 
    label=\makebox[2em]{\centering\arabic*$^\circ$} , 
}

\setlist[itemize, 2, 3, 4, 5]{
    labelindent=0em, 
    leftmargin=\parindent, 
    labelsep=0pt, 
}

\newlist{enumerate-abc}{enumerate}{3}
\setlist[enumerate-abc]{
    labelindent=0em, 
    leftmargin=3\parindent, 
    labelsep=0pt, 
    label=\makebox[2em]{\centering(\alph*)}
}

\newcommand{\figConstructionOfDeductionTree}{%
    \begin{figure}[htbp]
        \centering
        \begin{subfigure}{0.9\textwidth}
            \centering
            \begin{tikzpicture}[
                level distance=40pt, 
                sibling distance=20pt  
            ]%
                \Tree[ .{$\calH\mid \Gamma\tto B\to C, \Delta$}
                    {$\calH\mid \Gamma\tto B\to C, \Delta\mid \Gamma\tto\Delta$~~~~~}
                    \edge node[auto] {\small stage $B\to C$}; 
                    {~~~~~$\calH\mid \Gamma\tto B\to C, \Delta\mid \Gamma, B\tto C, \Delta$}
                ]
            \end{tikzpicture}
            \caption{$\cdots \tto B\to C, \cdots$}
            \label{subfig:construction of deduction tree-⇒→}
        \end{subfigure}

        \vspace{1em}

        \begin{subfigure}{0.9\textwidth}
            \centering
            \begin{tikzpicture}[
                level distance=40pt, 
                sibling distance=20pt  
            ]%
                \Tree[ .{$\calH\mid \Gamma, B\to C\tto \Delta$}
                    \edge node[left] {\small stage $B\to C$}; 
                    {$\calH\mid \Gamma, B\to C\tto \Delta\mid \Gamma\tto \Delta \mid \Gamma, C\tto B, \Delta$}
                ]
            \end{tikzpicture}
            \caption{$\cdots, B\to C \tto \cdots$}
            \label{subfig:construction of deduction tree-→⇒}
        \end{subfigure}

        \vspace{1em}

        \begin{subfigure}{0.49\textwidth}
            \centering
            \begin{tikzpicture}[
                level distance=40pt, 
                sibling distance=20pt  
            ]%
                \Tree[ .{$\calH\mid\Gamma\tto \exists x B(x), \Delta$}
                \edge node[right] {\small ~~stage $(\exists x B(x), t_i)$}; 
                    {$\calH\mid\Gamma\tto \exists x B(x), \Delta\mid\Gamma\tto B(t_i), \Delta$}
                ]
            \end{tikzpicture}
            \caption{$\cdots \tto \exists x B(x), \cdots$}
            \label{subfig:construction of deduction tree-⇒∃}
        \end{subfigure}
        \hfill
        \begin{subfigure}{0.49\textwidth}
            \centering
            \begin{tikzpicture}[
                level distance=40pt, 
                sibling distance=20pt  
            ]%
                \Tree[ .{$\calH\mid \Gamma, \exists x B(x)\tto \Delta$}
                \edge node[left] {\small ~~stage $\exists x B(x)$}; 
                    {$\calH\mid \Gamma, \exists x B(x)\tto \Delta\mid \Gamma, B(a)\tto \Delta$}
                ]
            \end{tikzpicture}
            \caption{$\cdots, \exists x B(x)\tto \cdots$}
            \label{subfig:construction of deduction tree-∃⇒}
        \end{subfigure}

        \caption{Construction of $\DT_n(T, \calG)$}
        \label{fig:construction of deduction tree}
    \end{figure}
}

\newcommand{\figExampleOfRepeatedFormula}{%
    \begin{figure}[ht]
        \begin{tikzpicture}[
            scale=0.6, transform shape,
            level distance=46pt, 
            sibling distance=10pt  
        ]%
            \Tree
            [.{$\calH\mid \Gamma_1, B^1\to^1 C^1 \tto \Delta_1\mid \Gamma_2 \tto B^2\to^2 C^2, B^3\to^3 C^3, \Delta_2$}
                \edge node[left] {$B^1\to^1 C^1$}; 
                [ .{$\calG^{\sigma}\mid \Gamma_1 \tto \Delta_1\mid \Gamma_1, C^1 \tto B^1, \Delta_1$}
                    [ .{$\calG^{\sigma1}\mid \Gamma_2 \tto B^3\to^3 C^3, \Delta_2$}
                        {$\calG^{\sigma11}\mid \Gamma_2 \tto B^2\to^2 C^2, \Delta_2$}
                        \edge node[auto] {~~~~~~$B^3\to^3 C^3$}; 
                        {$\calG^{\sigma11}\mid \Gamma_2, B^3 \tto C^3, B^2\to^2 C^2, \Delta_2$}
                    ]
                    \edge node[auto] {$B^2\to^2 C^2$}; 
                    [ .{$\calG^{\sigma1}\mid \Gamma_2, B^2 \tto C^2, B^3\to^3 C^3, \Delta_2$}
                        {$\calG^{\sigma12}\mid \Gamma_2\tto B^2\to^2 C^2, \Delta_2$}
                        \edge node[auto] {~~~~~~$B^3\to^3 C^3$}; 
                        {$\calG^{\sigma12}\mid \Gamma_2, B^3 \tto C^3, B^2\to^2 C^2, \Delta_2$}
                    ]
                ]
            ]
        \end{tikzpicture}
        \caption{The subtree of $\sigma$ in Example~{\ref{eg: construction for the formula appearing many times}}}
        \label{fig:construction for the formula appearing many times}
    \end{figure}
}

\newcommand{\figDeductionTreeOfDrinkerParadox}{
    \begin{figure}[htb]
        \centering
        \hspace{6em}
        \begin{tikzpicture}[scale=1, transform shape, level distance=5em]
            \Tree
            [. $\tto \exists x\forall y(P(x)\to P(y))$ 
                \edge node[right] {~stage $(\exists x\forall y(P(x)\to P(y)), a)$};
                [. $\calG^\varepsilon \mid {{} \tto \forall y(P(a)\to P(y))}$ 
                    \edge node[right] {~stage $\forall y(P(a)\to P(y))$};
                    [. {$\calG^{1} \mid {{} \tto P(a)\to P(b)}$} 
                        \edge node[right] {~stage $(\exists x\forall y(P(x)\to P(y)), b)$};
                        [. {$\calG^{11} \mid {{} \tto \forall y(P(b)\to P(y))}$} 
                            \edge node[right] {~stage $\forall y(P(b)\to P(y))$};
                            [. {$\calG^{111} \mid {{} \tto P(b)\to P(c)}$} 
                            ]
                        ]
                    ]
                ]
            ]
        \end{tikzpicture}
        \caption{Deduction tree $\DT_2(\varnothing, {}\tto \exists x\forall y(P(x)\to P(y)))$}
        \label{fig:example of 2-deduction tree of drinker's paradox}
    \end{figure}
}

\newcommand{\figOriginalDerivationOfDrinkerParadox}{
    \begin{figure}[htb]
        \centering
        \AxiomC{$\cdots$}
        \RightLabel{\makebox[1em]{\hspace{8.5em} Completeness of $\GL$}}
        \UnaryInfC{$\prop(\calG^{1111})_{\frac12}$}
        \RightLabel{(EW)}
        \UnaryInfC{$\calG^{1111}_{\frac12}$}
        \RightLabel{($\tto_{\frac12}{\forall}$)}
        \UnaryInfC{$\calG^{111}_{\frac12}\mid {{} \tto_{\frac12} \forall y(P(b)\to P(y))}$}
        \RightLabel{(EC)}
        \UnaryInfC{$\calG^{111}_{\frac12}$}
        \RightLabel{($\tto_{\frac12}{\exists}$)}
        \UnaryInfC{$\calG^{11}_{\frac12} \mid {{} \tto_{\frac12} \exists x\forall y(P(x)\to P(y))}$}
        \RightLabel{(EC)}
        \UnaryInfC{$\calG^{11}_{\frac12}$}
        \RightLabel{($\tto_{\frac12}{\forall}$)}
        \UnaryInfC{$\calG^1_{\frac12} \mid {{} \tto_{\frac12} \forall y (P(a)\to P(y))}$}
        \RightLabel{(EC)}
        \UnaryInfC{$\calG^1_{\frac12}$}
        \RightLabel{($\tto_{\frac12}{\exists}$)}
        \UnaryInfC{$\calG^{\varepsilon}_{\frac12} \mid {{} \tto_{\frac12} \exists x\forall y(P(x)\to P(y))}$}
        \RightLabel{(EC)}
        \UnaryInfC{$\calG^{\varepsilon}_{\frac12}$}
        \DisplayProof
        \caption{A derivation of ${}\tto_{\frac12} \exists x\forall y(P(x)\to P(y))$}
        \label{fig:strict derivation}
    \end{figure}
}

\newcommand{\figSimplifiedDerivationOfDrinkerParadox}{
    \begin{figure}[htb]
        \centering
        \AxiomC{$\cdots$}
        \RightLabel{Completeness of $\GL$}
        \UnaryInfC{${} \tto_{\frac12} P(a)\to P(b)\mid {}\tto_{\frac12} P(b)\to P(c)$}
        \RightLabel{($\tto_{\frac12}{\forall}$)}
        \UnaryInfC{${} \tto_{\frac12} P(a)\to P(b)\mid {}\tto_{\frac12} \forall y(P(b)\to P(y))$}
        \RightLabel{($\tto_{\frac12}{\exists}$)}
        \UnaryInfC{${} \tto_{\frac12} P(a)\to P(b)\mid {}\tto_{\frac12} \exists x\forall y(P(x)\to P(y))$}
        \RightLabel{($\tto_{\frac12}{\forall}$)}
        \UnaryInfC{${}\tto_{\frac12} \forall y(P(a)\to P(y))\mid {}\tto_{\frac12} \exists x\forall y(P(x)\to P(y))$}
        \RightLabel{($\tto_{\frac12}{\exists}$)}
        \UnaryInfC{${}\tto_{\frac12} \exists x\forall y(P(x)\to P(y))\mid {}\tto_{\frac12} \exists x\forall y(P(x)\to P(y))$}
        \RightLabel{(EC)}
        \UnaryInfC{${}\tto_{\frac12} \exists x\forall y(P(x)\to P(y))$}
        \DisplayProof
        \caption{A simplified derivation of ${}\tto_{\frac12} \exists x\forall y(P(x)\to P(y))$}
        \label{fig:simplified derivation}
    \end{figure}
}
\newcommand{\nn}{\mathbb{N}}

\newcommand{\rr}{\mathbb{R}}

\newcommand{\tto}{\mathrel{\Rightarrow}}

\newcommand{\otto}{\mathrel{\Leftrightarrow}}

\newcommand{\calL}{\mathcal{L}}

\newcommand{\calG}{\mathcal{G}}
\newcommand{\calH}{\mathcal{H}}
\newcommand{\calS}{\mathcal{S}}
\newcommand{\calM}{\mathcal{M}}

\newcommand{\frakI}{\mathfrak{I}}

\newcommand{\prop}{\operatorname{prop}}

\newcommand{\HL}{\text{\upshape\rmfamily HŁ}}
\newcommand{\GL}{\text{\upshape\rmfamily GŁ}}

\newcommand{\cp}{\operatorname{cp}}
\newcommand{\DT}{\operatorname{DT}}

\usepackage{stmaryrd}
\newcommand{\fs}[1]{\left\llbracket#1\right\rrbracket}

\title[The approximate strong completeness of $\GL\forall$]{The approximate strong completeness of the hypersequent calculus $\GL\forall$}
\author{Kai Duo~\orcidlink{0009-0003-6027-9728}}
\address{School of Mathematical Sciences, Beihang University, Beijing 102206, China}
\email{duokai@buaa.edu.cn}
\address{Beijing Institute of Mathematical Sciences and Applications, Beijing 101408, China}
\email{duokai@bimsa.cn}
\thanks{%
    The author would like to express sincere gratitude to Prof.~Kazuyuki~Tanaka, Prof.~Yichuan~Yang and Prof.~Wenjuan~Li for their invaluable guidance and support.
    The author would also like to thank Dr.~Jin~Wei for the helpful discussion.
}

\date{}



\begin{document}

\begin{abstract}
    The hypersequent calculus $\GL\forall$, an analytic Gentzen-style proof system of first-order Łukasiewicz logic, and its approximate completeness have been extensively studied.
    In this paper, we prove the approximate strong completeness of $\GL\forall$ by a labelled tableau method.
    Then we introduce a sequent-level cut rule (s-Cut) and show the approximate strong completeness of $\GL\forall+(\text{s-Cut})$.
    As applications, we establish the compactness of approximate $[0, 1]$-consequence, a variant of Gentzen's mid-sequent theorem of $\GL\forall$ and an approximate Herbrand's theorem of first-order Łukasiewicz logic.
\end{abstract}

\maketitle

\section{Introduction}

Infinite-valued Łukasiewicz logic, first introduced in \cite{lukasiewicz1930}, is one of the central and most extensively studied many-valued logics.
It is regarded, together with Gödel logic and Product logic, as one of the fundamental t-norm-based fuzzy logics; see \cite{hajek2001} for a general account.
The first-order Łukasiewicz logic under $[0, 1]$-semantics extends classical first-order logic by interpreting the quantifier symbols $\forall$ and $\exists$ by suprema and infima over $[0, 1]$ respectively.
Besides $[0, 1]$-semantics, Łukasiewicz logic also admits a so-called ``general semantics'' obtained by replacing $[0, 1]$ with an arbitrary MV-algebra.
Introduced by Chang~\cite{chang1958}, MV-algebras provide the algebraic semantics of Łukasiewicz logic.
The algebraic study on MV-algebras has been extensively developed; the monograph \cite{cignoli2013} is a standard reference.
In this paper, we work mainly with $[0, 1]$-semantics.

There are two Hilbert-style proof systems for Łukasiewicz logic: $\HL$ for propositional case and $\HL\forall$ for first-order case.
Chang~\cite{chang1958} first established the completeness of $\HL$ via the celebrated theorem that an MV-equation holds in every MV-algebra if and only if it holds in the standard MV-algebra $[0, 1]$, and hence the general semantics coincides with the $[0, 1]$-semantics.
However, this is not the case in the first-order setting.
$\HL\forall$ satisfies strong completeness under the general semantics \cite[Theorems 5.2.9]{hajek2001}, whereas under the $[0, 1]$-semantics it only satisfies approximate strong completeness \cite[Theorem 5.4.25]{hajek2001}, which asserts that if $T\models A$, then $\HL\forall$-derivations from $T$ can provide arbitrary good approximations of the truth value of $A$.
The approximation is necessary, since the set of $[0, 1]$-valid formulas is not recursively enumerable; more precisely, it is $\Pi_2$-complete, see \cite[Theorem 6.3.4]{hajek2001}.

Gentzen-style proof systems of Łukasiewicz logic are naturally formulated in terms of hypersequents, a generalization of two-sided sequents.
Hypersequents were introduced by Avron~\cite{avron1987} in the study of the relevance logic RM.
The first hypersequent calculus for a fuzzy logic was proposed in \cite{avron1991} for Gödel logic, and since then a variety of hypersequent calculi have been developed for many-valued and fuzzy logics; see \cite{metcalfe2008} for more details.
For propositional Łukasiewicz logic, the first hypersequent calculus $\GL$ was introduced by Metcalfe et al.~\cite{metcalfe2005}, who established the soundness and completeness of $\GL$.

Baaz and Metcalfe~\cite{baaz2010} studied hypersequent calculus $\GL\forall$, a natural first-order extension of $\GL$.
Using Skolemization together with an approximate form of Herbrand's theorem, they proved the approximate completeness under $[0, 1]$-semantics for all formulas.
Subsequently, Gerasimov identified some gaps in their proof and gave a rigorous proof for all prenex formulas, then extended the result to all formulas by a deprenexification method; see \cite{gerasimov2020, gerasimov2026}.
Recently, Wei~\cite{wei2025} established the approximate completeness for all hypersequents.
These completeness results only concern validity with respect to all interpretations, rather than validity with respect to all models of a given theory; the latter is usually referred to as ``strong completeness''.
The approximate strong completeness of $\HL\forall$ suggests a corresponding result for $\GL\forall$, which remains to be established and constitutes the main motivation of this paper.


The main purpose of this paper is to prove an approximate strong completeness theorem for $\GL\forall$.
More precisely, we show that if a hypersequent $\calG$ is a $[0, 1]$-consequence of a theory $T$, then for each positive natural number $n$ one can derive a $1/n$-approximation of $\calG$ witnessed by only finitely many formulas from $T$ in $\GL\forall$.
Our proof is based on a labelled tableau construction inspired by Buss~\cite{buss1998}.

The remainder of this paper is organized as follows.
In Section~\ref{sec:preliminary}, we introduce the affine $[0, 1]$-semantics of hypersequents and hypersequent calculi $\GL$ and $\GL\forall$.
In Section~\ref{sec:approximate strong completeness}, we prove our main theorem, the approximate strong completeness of $\GL\forall$ (Theorem~\ref{thm:approximate strong completeness of GŁ∀}).
Then we introduce a sequent-level cut rule (s-Cut) and show the approximate strong completeness of $\GL\forall+(\text{s-Cut})$ (Theorem~\ref{thm:approximate strong completeness of GŁ∀+(s-Cut)}).
In Section~\ref{sec:applications}, we prove the compactness of approximate $[0, 1]$-consequence (Theorem~\ref{thm:compactness of approximate logical entailment}), a mid-hypersequent theorem (Theorem~\ref{thm:midhypersequent}) and an approximate Herbrand's theorem (Theorem~\ref{thm:approximate Herbrand's theorem}).
Finally, we discuss some potential applications.
\section{Preliminary}
\label{sec:preliminary}

In this section, we introduce affine $[0, 1]$-semantics of hypersequents and hypersequent calculi $\GL$ and $\GL\forall$.
The reader is referred to \cite{buss1998} for the terminology of proof theory.

\subsection{Propositional Łukasiewicz logic and hypersequent calculus $\GL$}
\label{subsec:GŁ}
The \emph{formula}s of propositional Łukasiewicz logic are built from an infinite set $X$ of propositional variables with two connectives $\bot$ ($0$-ary) and $\to$ (binary).
A \emph{sequent} $\calS$ is an ordered pair $(\Gamma, \Delta)$ of finite multisets of formulas, denoted by $\Gamma\tto \Delta$.
A \emph{hypersequent} $\calG$ is a finite multiset $\{\calS_1, \cdots , \calS_n\}$ of sequents, denoted by $\calS_1\mid \cdots \mid \calS_n$.
Each $\calS_i$ is called a \emph{component} sequent of $\calG$.
We will also use the term ``infinite hypersequent'' to refer to an infinite multiset of sequents.
However, unless otherwise specified, hypersequents are always assumed to be finite.

We introduce the following notations.
For each $n\in\nn_{>0}$, formula $A$ and finite multiset $\Gamma$ of formulas, $n A$ denotes the multiset $\{A, \cdots, A\}$ ($n$ times) and  $n\Gamma$ denotes the multiset $\Gamma\sqcup\cdots\sqcup\Gamma$ ($n$ times), where $\sqcup$ denotes the disjoint union of multisets.
Sequent $\{A_1, \cdots, A_n\}\tto \{B_1, \cdots, B_m\}$ is usually abbreviated as $A_1, \cdots, A_n\tto B_1, \cdots, B_m$.
We write $\Gamma \tto {}$ as an abbreviation for $\Gamma \tto \varnothing$; the cases of ${} \tto \Delta$ and ${}\tto{}$ are similar.

Now we define affine $[0, 1]$-semantics of hypersequents.
A $[0, 1]$-\emph{evaluation} is a map $v: X\to [0, 1]$.
We extend $v$ to a function $\tilde{v}$ as follows:
\begin{enumerate-abc}
    \item $\tilde{v}(\bot)=1$.
    \item $\tilde{v}(A)=v(A)$ for every propositional variable $A\in X$.
    \item $\tilde{v}(A\to B)=\max\{\tilde{v}(B)-\tilde{v}(A), 0\}$.
    \item $\tilde{v}(\Gamma)=\sum_{A\in\Gamma} \tilde{v}(A)$ for finite $\Gamma\neq\varnothing$.
    $\tilde{v}(\varnothing)=0$.
    \item $\tilde{v}(\Gamma\tto \Delta)=\tilde{v}(\Delta)-\tilde{v}(\Gamma)$.
    \item For each hypersequent $\calG$ (possibly infinite), $\tilde{v}(\calG)=\inf_{\calS\in \calG} \tilde{v}(\calS)$.
\end{enumerate-abc}

We say a hypersequent $\calG$ is \emph{$[0, 1]$-valid}, denoted by $\models_{[0, 1]} \calG$, if $\tilde{v}(\calG)\leq 0$ for any $[0, 1]$-evaluation $v$.

We give some remarks here.

\begin{remark}~
    \begin{enumerate}
        \item For each $[0, 1]$-evaluation $v$, formula $A$ and hypersequent $\calG$, $\tilde{v}(A)\in [0, 1]$ while $\tilde{v}(\calG)\in\rr$.
        \item We use $0$ to represent truth and $1$ to represent falsity.
        This convention is also widely adopted in continuous model theory (see \cite{benyaacov2008, benyaacov2010}), and makes the subsequent representation of approximate completeness more convenient.
        \item In the definition of $\tilde{v}$, clause (d) explains the term ``affine'' in ``affine $[0, 1]$-semantics''; it also shows intuitively that both sides of a sequent are interpreted conjunctively, while clause (f) indicates that a hypersequent is interpreted as the disjunction of its component sequents.
        \item A formula $A$ is said to be $[0, 1]$-valid if $\tilde{v}(A)=0$ for any $[0, 1]$-evaluation $v$.
        It is easy to see that the $[0, 1]$-validity of $A$ is equivalent to the $[0, 1]$-validity of the hypersequent ${}\tto A$.
        Therefore, each formula $A$ can be identified with hypersequent ${}\tto A$, and affine $[0, 1]$-semantics of hypersequents can be viewed as an extension of $[0, 1]$-semantics of formulas in this sense.
        \item It is also possible to consider affine $\mathbf{A}$-semantics for each MV-algebra $\mathbf{A}$ by using Mundici's equivalence, which states that the category of MV-algebras and the category of abelian $\ell$-groups with strong unit are equivalent; see \cite{mundici1986}.
        This yields the general semantics; see \cite{baaz2010} for details.
    \end{enumerate}
\end{remark}

Next we define the hypersequent calculus $\GL$.

\begin{definition}[\cite{metcalfe2005}]
    $\GL$ has the following rules: 
    \begin{itemize}
        \item Axioms: 
        \begin{center}
            \scalebox{0.9}{
            \AxiomC{}
            \RightLabel{(ID)}
            \UnaryInfC{$A\tto A$}
            \DisplayProof
            \qquad
            \AxiomC{}
            \RightLabel{(EMP)}
            \UnaryInfC{$\tto$}
            \DisplayProof
            \qquad
            \AxiomC{}
            \RightLabel{($\bot$)}
            \UnaryInfC{$\bot\tto A$}
            \DisplayProof
            }
        \end{center}
        \item Structural rules: 
        \begin{center}
            \scalebox{0.9}{
            \AxiomC{$\calG$}
            \RightLabel{\makebox[2em][l]{(EW)}}
            \UnaryInfC{$\calG\mid \Gamma\tto\Delta$}
            \DisplayProof
            \qquad
            \AxiomC{$\calG\mid \Gamma\tto\Delta\mid \Gamma\tto\Delta$}
            \RightLabel{\makebox[2em][l]{(EC)}}
            \UnaryInfC{$\calG\mid \Gamma\tto\Delta$}
            \DisplayProof
            \qquad
            \AxiomC{$\calG\mid \Gamma\tto\Delta$}
            \RightLabel{\makebox[2em][l]{(IW)}}
            \UnaryInfC{$\calG\mid \Gamma, A\tto\Delta$}
            \DisplayProof
            }\vspace{0.3em}
            \scalebox{0.9}{
            \AxiomC{$\calG\mid \Gamma_1, \Gamma_2\tto\Delta_1, \Delta_2$}
            \RightLabel{\makebox[2em][l]{(S)}}
            \UnaryInfC{$\calG\mid \Gamma_1\tto\Delta_1\mid\Gamma_2\tto\Delta_2$}
            \DisplayProof
            \qquad
            \AxiomC{$\calG\mid \Gamma_1\tto\Delta_1$}
            \AxiomC{$\calG\mid \Gamma_2\tto\Delta_2$}
            \RightLabel{\makebox[2em][l]{(M)}}
            \BinaryInfC{$\calG\mid \Gamma_1, \Gamma_2\tto\Delta_1, \Delta_2$}
            \DisplayProof
            }
        \end{center}
        \item Logical rules: 
        \begin{center}
            \scalebox{0.9}{
            \AxiomC{$\calG\mid \Gamma\tto\Delta\mid \Gamma, B\tto A, \Delta$}
            \RightLabel{\makebox[2em][l]{(${\to}\tto$)}}
            \UnaryInfC{$\calG\mid \Gamma, A\to B\tto\Delta$}
            \DisplayProof
            \qquad
            \AxiomC{$\calG\mid \Gamma\tto\Delta$}
            \AxiomC{$\calG\mid \Gamma, A\tto B, \Delta$}
            \RightLabel{\makebox[2em][l]{($\tto{\to}$)}}
            \BinaryInfC{$\calG\mid \Gamma\tto A\to B, \Delta$}
            \DisplayProof
            }
        \end{center}
    \end{itemize}
    In (EW), we assume that $\calG\neq\varnothing$.
\end{definition}

It is easy to see that $\GL$ is an analytic proof system, i.e., every $\GL$-derivation satisfies ``subformula property''.
Hypersequent calculus $\GL$ enjoys soundness and completeness.

\begin{theorem}[\cite{metcalfe2005, wei2025}]\label{thm:completeness of GŁ}
    Let $\calG$ be a hypersequent.
    Then $\vdash_{\GL} \calG$ if and only if $\models_{[0, 1]} \calG$.
\end{theorem}

\subsection{First-order Łukasiewicz logic and hypersequent calculus $\GL\forall$}
\label{sec:GŁ∀}
For a technical reason---namely, to guarantee the substitutability of terms for variables, a standard requirement in proof theory---we distinguish between free variables and bound variables, semiterms and terms, and so on.
We follow the same convention as in \cite{buss1998}.
Two disjoint infinite sets of variables are fixed, called \emph{free variable}s and \emph{bound variable}s.
Free variables are denoted by metavariables $a, b, c, \cdots$ and bound variables by metavariables $x, y, z, \cdots$.

From now on, we fix a first-order language $\calL$, which includes predicate symbols, function symbols, constant symbols, together with their arities.
\emph{$\calL$-term}s are built from free variables, constant symbols and function symbols in a usual way, while \emph{$\calL$-semiterm}s are built from all variables, constant symbols and function symbols.
An \emph{atomic} $\calL$-formula (respectively, \emph{atomic} $\calL$-semiformula) is of the form $P(t_1, \cdots, t_k)$, where $P\in\calL$ is a predicate symbol with arity $k$ and $t_1, \cdots, t_k$ are $\calL$-terms (respectively, $\calL$-semiterms).
\emph{$\calL$-formulas} are defined inductively as follows: 
\begin{enumerate-abc}
    \item $\bot$ is an $\calL$-formula.
    Each atomic $\calL$-formula is an $\calL$-formula.
    \item If both $A$ and $B$ are $\calL$-formulas, so is $A\to B$.
    \item If $A$ is an $\calL$-formula, $a$ is a free variable and $x$ is a bound variable not occurring in $A$, then $\forall x A[x/a]$ and $\exists x A[x/a]$ are $\calL$-formulas, where $A[x/a]$ is the expression obtained from $A$ by replacing every occurrence of $a$ in $A$ with $x$.
    \item $\calL$-formulas are obtained in finitely many steps by applying only (a)-(c).
\end{enumerate-abc}
An \emph{$\calL$-semiformula} is an expression like an $\calL$-formula, except that bound variables are also allowed to occur free in it.
Note that free variables can never be captured by quantifiers.
\emph{$\calL$-(semi)sequent}s and \emph{$\calL$-(semi)hypersequent}s are also defined same as in the propositional case, except replacing ``formulas'' by ``$\calL$-(semi)formulas''.

We define affine $[0, 1]$-semantics of $\calL$-semihypersequents.
A $[0, 1]$-\emph{structure} $\calM$ in language $\calL$ is defined as usual, except that the interpretation of an $n$-ary predicate symbol $P$ is a function $P^{\calM}: M^n\to [0, 1]$.
An $\calM$-\emph{assignment} $\beta$ is a function that assigns an element in $M$ to each free variable and bound variable.
We can extend $\beta$ to the set of $\calL$-semiterms inductively as follows: 
\begin{enumerate-abc}
    \item $\beta(e)=e^{\calM}$, where $e$ is a constant symbol.
    \item $\beta(f(t_1, \cdots, t_n))=f^{\calM}(\beta(t_1), \cdots, \beta(t_n))$, where $f$ is an $n$-ary function symbol and $t_1, \cdots, t_n$ are $\calL$-semiterms.
\end{enumerate-abc}
A $[0, 1]$-\emph{interpretation} $\frakI$ is an ordered pair $(\calM, \beta)$, where $\calM$ is a $[0, 1]$-structure and $\beta$ is an $\calM$-assignment.
Let $\frakI=(\calM, \beta)$ be a $[0, 1]$-interpretation, $x_0$ is a variable (possibly either free or bound) and $m\in M$.
Then we write $\frakI[x_0\mapsto m]$ for the $[0, 1]$-interpretation $(\calM, \beta[x_0\mapsto m])$, where $\beta[x_0\mapsto m]$ is defined by 
\begin{equation*}
    \beta[x_0\mapsto m](x)=
    \begin{cases}
        m, &\text{if $x$ is $x_0$}\\
        \beta(x), &\text{if $x$ is not $x_0$}
    \end{cases}.
\end{equation*}

Let $\frakI=(\calM, \beta)$ be a $[0, 1]$-interpretation.
For each $\calL$-semihypersequent $\calG$, the \emph{truth value} $\fs{\calG}_{\frakI}$ of $\calG$ with respect to $\frakI$ is defined inductively as follows: 
\begin{enumerate-abc}
    \item $\fs{\bot}_{\frakI}=1$.
    \item $\fs{P(t_1, \cdots, t_n)}_{\frakI}=P^{\calM}(\beta(t_1), \cdots, \beta(t_n))$.
    \item $\fs{A\to B}_{\frakI}=\max\{\fs{B}_{\frakI}-\fs{A}_{\frakI}, 0\}$.
    \item $\fs{\forall xA}_{\frakI}=\sup\{\fs{A}_{\frakI[x\mapsto m]}: m\in M\}$.
    \item $\fs{\exists xA}_{\frakI}=\inf\{\fs{A}_{\frakI[x\mapsto m]}: m\in M\}$.
    \item $\fs{\Gamma}_{\frakI}=\sum_{A\in\Gamma} \fs{A}_{\frakI}$ for finite $\Gamma\neq\varnothing$.
    $\fs{\varnothing}_{\frakI}=0$.
    \item $\fs{\Gamma\tto \Delta}_{\frakI}=\fs{\Delta}_{\frakI}-\fs{\Gamma}_{\frakI}$.
    \item For each $\calL$-semihypersequent $\calG$ (possibly infinite), $\fs{\calG}_{\frakI}=\inf_{\calS\in \calG} \fs{\calS}_{\frakI}$.
\end{enumerate-abc}

We still identify each $\calL$-semiformula $A$ with $\calL$-semihypersequent ${}\tto A$.
An \emph{$\calL$-theory} is a set of $\calL$-formulas.
Let $T$ be an $\calL$-theory, $\calG$ be an $\calL$-semihypersequent and $\frakI$ be a $[0, 1]$-interpretation.
We introduce the following definitions:

\begin{itemize}
    \item $\calG$ is \emph{true} in $\frakI$, denoted by $\frakI\models_{[0, 1]} \calG$, if $\fs{\calG}_{\frakI}\leq 0$.
    \item $\frakI$ is a $[0, 1]$-\emph{model} of $T$ if every formula in $T$ is true in $\frakI$.
    \item $\calG$ is a \emph{$[0, 1]$-consequence} of $T$, denoted by $T\models_{[0, 1]} \calG$, if $\calG$ is true in every $[0, 1]$-model of $T$.
    \item $\calG$ is \emph{$[0, 1]$-valid} if $\models_{[0, 1]} \calG$, that is, $\calG$ is true in every $[0, 1]$-interpretation.
\end{itemize}

As in classical fist-order logic, substitution is important for connecting syntax and semantics.
We shall say that ``$A(x)$ is an $\calL$-(semi)formula'', meaning that $A$ is an $\calL$-(semi)formula with a distinguished variable $x$ (possibly either free or bound), which may or may not occur in $A$.
For each $\calL$-semiformula $A(x)$ and $\calL$-semiterm $t$, $A(t)$ is the $\calL$-semiformula obtained by replacing every free occurrence of $x$ in $A$ with $t$.
We say $t$ is \emph{substitutible} for $x$ in $A(x)$ if no bound variable $y$ in $t$ is captured by a quantifier $\forall y$ or $\exists y$ in $A(t)$.
Note that $t$ is always substitutible for $x$ in $A(x)$ for each $\calL$-term $t$ and $\calL$-formula of the form $\forall x A(x)$ (or $\exists x A(x)$).
The next lemma is an analogue of the substitution lemma for classical first-order logic.
A corresponding version for continuous logic can be found in \cite{benyaacov2010}.

\begin{proposition}[Substitution lemma]\label{prop:substitution lemma}
    Let $\frakI=(\calM, \beta)$ be a $[0, 1]$-interpretation.
    Then the following statements hold:
    \begin{enumerate}
        \item $\beta[x\mapsto \beta(t)](s) = \beta(s[t/x])$ for each $\calL$-semiterm $s$ and $t$.
        \item Let $A(x)$ be an $\calL$-semiformula and $t$ be an $\calL$-semiterm.
        If $t$ is substitutible for $x$ in $A(x)$, then $\fs{A(x)}_{\frakI[x\mapsto \beta(t)]}=\fs{A(t)}_{\frakI}$.
    \end{enumerate}
\end{proposition}






We define the hypersequent calculus $\GL\forall$.

\begin{definition}[\cite{baaz2010}]
    Let $\GL\forall$ be $\GL$ with the added logical rules for quantifiers: 
    \begin{center}
        \scalebox{0.9}{
            \AxiomC{$\calG\mid \Gamma, A(t)\tto\Delta$}
            \RightLabel{($\forall\tto$)}
            \UnaryInfC{$\calG\mid \Gamma, \forall xA(x)\tto\Delta$}
            \DisplayProof
            \qquad
            \AxiomC{$\calG\mid \Gamma\tto A(a), \Delta$}
            \RightLabel{($\tto\forall$)}
            \UnaryInfC{$\calG\mid \Gamma\tto \forall xA(x), \Delta$}
            \DisplayProof
        }\vspace{0.3em}
        \scalebox{0.9}{
            \AxiomC{$\calG\mid \Gamma, A(a)\tto\Delta$}
            \RightLabel{($\exists\tto$)}
            \UnaryInfC{$\calG\mid \Gamma, \exists xA(x)\tto\Delta$}
            \DisplayProof
            \qquad
            \AxiomC{$\calG\mid \Gamma\tto A(t), \Delta$}
            \RightLabel{($\tto\exists$)}
            \UnaryInfC{$\calG\mid \Gamma\tto \exists xA(x), \Delta$}
            \DisplayProof
        }
    \end{center}
    where $t$ is any $\calL$-term and $a$ is a free variable which does not occur in the conclusion of ($\tto\forall$) or ($\exists\tto$); such a free variable $a$ is called the \emph{eigenvariable} of the rule.
\end{definition}

Although derivations were already mentioned in the previous subsection, we nevertheless give precise definitions of derivations and derivations from a theory here to avoid any possible ambiguity.
Note that we restrict ourselves to derivations formed by formulas, rather than semiformulas.

\begin{definition}\label{def:derivation}
    Let $\text{\upshape\rmfamily S}$ be a hypersequent calculus containing $\GL\forall$.
    \begin{itemize}
        \item An \emph{$\text{\upshape\rmfamily S}$-derivation} is a finite tree $\pi$ labelled by $\calL$-hypersequents whose leaves are labelled by axioms of $\text{\upshape\rmfamily S}$, and labels of other nodes are obtained by applying inference rules in $\text{\upshape\rmfamily S}$.
        If the root of $\pi$ is labelled by $\calG$, then we say $\calG$ is $\text{\upshape\rmfamily S}$-\emph{derivable}, denoted by $\vdash_{\text{\upshape\rmfamily S}} \calG$, and $\pi$ is an $\text{\upshape\rmfamily S}$-derivation of $\calG$.
        \item Let $T$ be an $\calL$-theory.
        An \emph{$\text{\upshape\rmfamily S}_{T}$-derivation} is an $\text{\upshape\rmfamily S}$-derivation $\pi$ satisfies the following conditions:
        \begin{enumerate-abc}
            \item For each $C\in T$, sequent ${}\tto C$ can be used as an axiom in $\pi$;
            \item eigenvariables used in $\pi$ do not occur in $T$.
        \end{enumerate-abc}
        We say an $\calL$-hypersequent $\calG$ is \emph{derivable from $T$ in $\text{\upshape\rmfamily S}$}, denoted by $T\vdash_{\text{\upshape\rmfamily S}} \calG$, if there is an $\text{\upshape\rmfamily S}_{T}$-derivation of $\calG$.
    \end{itemize}
\end{definition}

Hypersequent calculus $\GL\forall$ is also analytic, and sound with respect to affine $[0, 1]$-semantics.
In fact, we can show the following strong soundness result.

\begin{theorem}\label{thm:strong soundness}
    Let $T$ be an $\calL$-theory and $\calG$ be an $\calL$-hypersequent.
    If $T\vdash_{\GL\forall} \calG$, then $T\models_{[0, 1]} \calG$.
\end{theorem}

\begin{proof}
    Strong soundness can be proved directly by induction on the height of the derivations, based on the following statements: 
    \begin{enumerate-abc}
        \item Axioms in $\GL\forall$ and formulas in $T$ are true in every $[0, 1]$-model of $T$. 
        \item For each inference rule in $\GL\forall$, if its premises are true in every $[0, 1]$-model of $T$, so is its conclusion.
        If the inference rule is ($\tto\forall$) or ($\exists\tto$), we additionally assume that the eigenvariable does not occur in $T$ (see Definition~\ref{def:derivation} (b) for the motivation).
    \end{enumerate-abc}

    Now we show above statements.
    Statement (a) is trivial.
    For statement (b), we only treat the case where the inference rule is ($\tto\forall$); the remaining cases are similar.
    Assume that $\calG\mid \Gamma\tto A(a), \Delta$ is true in every $[0, 1]$-model of $T$, and free variable $a$ does not occur in $T$, $\calG$, $\Gamma$ or $\Delta$.
    Let $\frakI=(\calM, \beta)$ be a $[0, 1]$-model of $T$.
    For simplicity, we write $\fs{~}_{\frakI}$ as $\fs{~}$ and write $\fs{~}_{\frakI[a\mapsto m]}$ as $\fs{~}_{m}$ for each $m\in M$.
    By our assumption on $a$, for each $m\in M$, we have $\fs{C}_{m}=\fs{C}=0$ for each $C\in T$, $\fs{\calG}_{m}=\fs{\calG}$, $\fs{\Gamma}_{m}=\fs{\Gamma}$ and $\fs{\Delta}_{m}=\fs{\Delta}$.
    Hence $\frakI[a\mapsto m]$ is a $[0, 1]$-model of $T$ and $\calG\mid \Gamma\tto A(a), \Delta$ is true in $\frakI[a\mapsto m]$ for each $m\in M$.
    By substitution lemma (Proposition~\ref{prop:substitution lemma}),  we have
    \begin{align*}
        \fs{\calG\mid \Gamma\tto \forall xA(x), \Delta}
        &=\min\{\fs{\calG}, -\fs{\Gamma}+ \fs{\forall xA(x)}+\fs{\Delta}\}\\
        &=\min\left\{\fs{\calG}, -\fs{\Gamma}+ \sup_{m\in M}\fs{A(x)}_{\frakI[x\mapsto m]}+\fs{\Delta}\right\}\\
        &=\min\left\{\fs{\calG}, -\fs{\Gamma}+ \sup_{m\in M}\fs{A(x)}_{\frakI[a\mapsto m][x\mapsto m]}+\fs{\Delta}\right\}\\
        &=\min\left\{\fs{\calG}, -\fs{\Gamma}+ \sup_{m\in M}\fs{A(a)}_m+\fs{\Delta}\right\}\\
        &=\sup_{m\in M}\min\{\fs{\calG}, -\fs{\Gamma}+ \fs{A(a)}_m+\fs{\Delta}\}\\
        &=\sup_{m\in M}\min\{\fs{\calG}_m, -\fs{\Gamma}_m+ \fs{A(a)}_m+\fs{\Delta}_m\}\\
        &=\sup_{m\in M}\fs{\calG\mid \Gamma\tto A(a), \Delta}_m\leq 0.
    \end{align*}
    Hence $\calG\mid \Gamma\tto \forall xA(x), \Delta$ is true in $\frakI$.
    This completes the proof.
\end{proof}

However, as mentioned in Introduction, the set of $[0, 1]$-valid $\calL$-formulas is $\Pi_2$-complete; see \cite[Theorem 6.3.4]{hajek2001}.
Consequently, the set of $[0, 1]$-valid $\calL$-hypersequents is at least $\Pi_2$-hard, and in fact $\Pi_2$-complete as well, since Theorem~\ref{thm:approximate completeness of GŁ∀} shows that it belongs to $\Pi_2$.
Hence we cannot expect to obtain a completeness result in the context of the affine $[0, 1]$-semantics.
Approximate completeness needs to be considered, and for this purpose we introduce some convenient notations, which also illustrate the advantage of hypersequents.
For each  $n\in\nn_{>0}$, $\calL$-semisequent $\Gamma\tto \Delta$ and $\calL$-semihypersequent $\calG=\Gamma_1\tto\Delta_1\mid\cdots\mid \Gamma_k\tto \Delta_k$, we use $\Gamma\tto_{\frac{1}{n}} \Delta$ to denote the $\calL$-semisequent $\bot, n\Gamma\tto n\Delta$ and use $\calG_{\frac{1}{n}}$ to denote the $\calL$-semihypersequent $\Gamma_1\tto_{\frac1n}\Delta_1\mid\cdots\mid \Gamma_k\tto_{\frac1n} \Delta_k$.
By direct calculation, $\fs{\calG_{\frac1n}}_{\frakI}=n\fs{\calG}_{\frakI}-1$ for each $\calL$-semihypersequent $\calG$ and $[0, 1]$-interpretation $\frakI$.
Hence $\fs{\calG_{\frac1n}}_{\frakI}\leq 0$ if and only if $\fs{\calG}_{\frakI}\leq \frac1n$.
In this sense, $\calG_{\frac1n}$ is an approximate of $\calG$.
We say $\calG$ is \emph{$\frac1n$-valid} if $\calG_{\frac1n}$ is $[0, 1]$-valid.
$\calG$ is $[0, 1]$-valid if and only if $\calG$ is $\frac1n$-valid for each $n\in\nn_{>0}$.

Baaz and Metcalfe~\cite{baaz2010} proved the approximate completeness of $\GL\forall$ for all formulas.
Subsequently, Gerasimov identified some gaps in their proof and gave a rigorous proof for all prenex formulas, then extended the result to all formulas (Theorem~\ref{thm:approximate completeness of GŁ∀, formula}) by a deprenexification method; see \cite[Theorem 7 and the subsequent discussion]{gerasimov2020} and \cite[Theorem 5.3]{gerasimov2026}.
Recently, Wei~\cite{wei2025} established the approximate completeness for all hypersequents (Theorem~\ref{thm:approximate completeness of GŁ∀}).

\begin{theorem}[{\cite[Theorem 5.3]{gerasimov2026}}]\label{thm:approximate completeness of GŁ∀, formula}
    Let $A$ be an $\calL$-formula.
    Then $\models_{[0, 1]} A$ if and only if $\vdash_{\GL\forall} {}\tto_{\frac1n} A$ for each $n\in\nn_{>0}$.
\end{theorem}

\begin{theorem}[{\cite[Theorem 4.4.15]{wei2025}}]\label{thm:approximate completeness of GŁ∀}
    Let $\calG$ be a valid $\calL$-hypersequent.
    Then $\vdash_{\GL\forall} \calG_{\frac1n} $ for each $n\in\nn_{>0}$.
\end{theorem}

\section[The approximate strong completeness of {\sffamily GŁ∀}]{The approximate strong completeness of $\GL\forall$}
\label{sec:approximate strong completeness}
As mentioned at the end of the previous section, several results on approximate completeness have already been established.
However, they only concern validity with respect to all interpretations, rather than validity with respect to all models of a given theory.
In this section, we fill this gap by proving an approximate strong completeness theorem for all hypersequents, which is also the main theorem of this paper:

\begin{theorem}[Approximate strong completeness of $\GL\forall$]\label{thm:approximate strong completeness of GŁ∀}
    Let $\calL$ be a countable language, $T$ be an $\calL$-theory and $\calG$ be an $\calL$-hypersequent.
    If $T\models_{[0, 1]} \calG$, then for each $n\in\nn_{>0}$, there exist $m_1, \cdots, m_k\in\nn_{>0}$ and $C_1, \cdots, C_k\in T$ such that
    \begin{equation*}
        \vdash_{\GL\forall} m_1C_1\tto_{\frac1n}\bot\mid\cdots \mid m_kC_k\tto_{\frac1n}\bot \mid \calG_{\frac1n}.
    \end{equation*}
\end{theorem}

\begin{remark}~
    \begin{enumerate}
        \item If necessary, we can require that $m_1=\cdots=m_k$.
        Let $m=\max_{1\leq i\leq k} m_i$.
        Then $mC_1\tto_{\frac1n}\bot\mid\cdots \mid mC_k\tto_{\frac1n}\bot \mid \calG_{\frac1n}$ is derivable from $m_1C_1\tto_{\frac1n}\bot\mid\cdots \mid m_kC_k\tto_{\frac1n}\bot \mid \calG_{\frac1n}$ by (IW).
        \item The validity of the hypersequent $m_1C_1\tto_{\frac1n}\bot\mid\cdots \mid m_kC_k\tto_{\frac1n}\bot \mid \calG_{\frac1n}$ can be stated as follows: For each $[0, 1]$-interpretation $\frakI$, if $\fs{C_i}_{\frakI}<\frac{n-1}{m_in}$ for each $1\leq i\leq k$, then $\fs{\calG}_{\frakI}\leq\frac1n$.
        Intuitively, for each $n\in\nn_{>0}$, the $\frac1n$-validity of $\calG$ is witnessed by a finite subset of $T$; cf. Theorem~\ref{thm:compactness of approximate logical entailment}.
    \end{enumerate}
\end{remark}

In the next two subsections, we present a proof of Theorem~\ref{thm:approximate strong completeness of GŁ∀}, which is inspired by \cite[Chapter 1, 3.3.7]{buss1998}.
We first introduce some notions.
An $\calL$-semihypersequent $\calG$ is said to be \emph{atomic} if each formula in $\calG$ is either atomic or $\bot$.
An $\calL$-semihypersequent $\calG$ is said to be \emph{propositional} if each formula in $\calG$ is quantifier-free.
For each $\calL$-semihypersequent $\calG$ (possibly infinite), denote by $\prop(\calG)$ the largest sub-multiset of $\calG$ consisting of all propositional $\calL$-semisequents.

The general idea of our proof is to construct an ``$n$th-deduction tree'' to search a $\GL\forall$-derivation of $\calG_{\frac1n}$ from the bottom up, working backwards from $\calG$ to $\frac1n$-valid propositional hypersequents.
Now we describe the construction of the tree.

\begin{definition}[Deduction tree]\label{def:deduction tree}
    Let $\calL$ be a countable language.
    For each $n\in\nn_{>0}$, $\calL$-theory $T$ and $\calL$-hypersequent $\calG$, we define a labelled binary tree $\DT_n(T, \calG)$, called the $n$th-\emph{deduction tree} of $\calG$ with respect to $T$.
    Each node $\sigma\in \DT_n(T, \calG)$ is labelled by an $\calL$-hypersequent $\calG^{\sigma}$.
    The construction of $\DT_n(T, \calG)$ proceeds in stages: at each stage, $\DT_n(T, \calG)$ will be modified (we keep the same name $\DT_n(T, \calG)$ for the new partially constructed $\DT_n(T, \calG)$).
    A leaf $\sigma$ is said to be \emph{active} if $\prop(\calG^\sigma)_{\frac1n}$ is not derivable in $\GL$, or equivalently, by soundness and completeness of $\GL$, if $\prop(\calG^\sigma)_{\frac1n}$ is not $[0, 1]$-valid.

    We assume that there are infinitely many free variables not occurring in $T$, and fix an enumeration $\{(A_i, t_i)\}_{i=1}^{+\infty}$ of all pairs of $\calL$-formulas and $\calL$-terms, satisfying the following conditions: for each $\calL$-formula $A$ and $\calL$-term $t$, there are infinitely many $i$ such that $(A, t)=(A_i, t_i)$.

    Initially, $\DT_n(T, \calG)$ consists of just one root $\varepsilon$ labelled by $\calG$.

    At $i$-th stage, we consider the pair $(A_i, t_i)$.
    If $A_i\in T$, we replace $\calG^{\sigma}$ by $lA_i\tto\bot \mid \calG^{\sigma}$ for each $\sigma\in \DT_n(T, \calG)$, where $l\in\nn_{>0}$ is the smallest number such that $lA_i\tto\bot\notin \calG^{\varepsilon}$.
    Then we do the following:
    \begin{itemize}
        \item If $A_i$ is atomic we do nothing and proceed to the next stage.
        \item If $A_i$ is of the form $B\to C$, 
        then 
        \begin{itemize}
            \item To each active leaf $\sigma$ with label $\calG^\sigma$ of the form $\calH\mid \Gamma\tto B\to C, \Delta$, we attach two children $\sigma1$ and $\sigma2$, with label $\calG^{\sigma1}=\calG^\sigma\mid \Gamma\tto \Delta$ and $\calG^{\sigma2}=\calG^\sigma\mid \Gamma, B\tto C, \Delta$ (See Figure~\ref{subfig:construction of deduction tree-⇒→}).
            \item To each active leaf $\sigma$ with label $\calG^\sigma$ of the form $\calH\mid \Gamma, B\to C\tto \Delta$, we attach one child $\sigma1$, with label $\calG^{\sigma1}=\calG^\sigma\mid \Gamma\tto \Delta\mid \Gamma, C\tto B, \Delta$ (See Figure~\ref{subfig:construction of deduction tree-→⇒}).
        \end{itemize}
        \item If $A_i$ is of the form $\exists x B(x)$, 
        then 
        \begin{itemize}
            \item To each active leaf $\sigma$ with label $\calG^\sigma$ of the form $\calH\mid \Gamma\tto \exists x B(x), \Delta$, we attach one child $\sigma1$, with label $\calG^{\sigma1}=\calG^\sigma\mid \Gamma\tto  B(t_i), \Delta$ (See Figure~\ref{subfig:construction of deduction tree-⇒∃}).
            We remark that this, together with the dual $\cdots, \forall x B(x)\tto\cdots$ case, is the only case where $t_i$ is used, and the definitions of terms and formulas ensure that $t_i$ is substitutible for $x$ in $B(x)$.
            \item To each active leaf $\sigma$ with label $\calG^\sigma$ of the form $\calH\mid \Gamma, \exists x B(x)\tto \Delta$, we attach one child $\sigma1$, with label $\calG^{\sigma1}=\calG^\sigma\mid \Gamma, B(a)\tto \Delta$, where $a$ is a fresh free variable, not used in $\DT_n(T, \calG)$ and $T$ yet (See Figure~\ref{subfig:construction of deduction tree-∃⇒}).
        \end{itemize}
        \item The case where $A_i$ is of the form $\forall x B(x)$ is dual to the previous case.
        \item If $A_i$ occurs more than one time in $\calG^\sigma$ for some active leaves $\sigma$, then we will carry out the construction on each occurrence of $A_i$ in $\calG^\sigma$ according to a certain order.
        We provide an example below (Example~\ref{eg: construction for the formula appearing many times}) to show how the construction works.
        \item If there is no active leaf remaining in $\DT_n(T, \calG)$, we finish the construction of $\DT_n(T, \calG)$; otherwise, we move to the next stage.
    \end{itemize}
\end{definition}

\figConstructionOfDeductionTree

\begin{remark}\label{rmk: construction of deduction tree}~
    \begin{enumerate}
        \item The construction of $\DT_n(T, \calG)$ may not halt.
        There are three cases: 
        \begin{itemize}
            \item If $T=\varnothing$ and $\calG$ is atomic, then $\DT_n(T, \calG)$ consists of only one root node $\varepsilon$ at the limit stage, with label $\calG^\varepsilon=\calG$.
            \item If $T$ is a non-empty set of atomic $\calL$-formulas and $\calG$ is atomic, then $\DT_n(T, \calG)$ consists of only one root node $\varepsilon$ at the limit stage, labelled by an infinite hypersequent.
            \item In other cases, $\DT_n(T, \calG)$ is infinite at the limit stage.
        \end{itemize}
        \item If $\tau$ is a child of $\sigma$, then $\calG^{\sigma}\subseteq\calG^{\tau}$.
    \end{enumerate}
\end{remark}

\begin{example}\label{eg: construction for the formula appearing many times}
    Suppose that we are constructing $\DT_n(T, \calG)$ at stage $(A_i=B\to C, t_i)$ and there is an active leaf $\sigma$ with label $\calG^\sigma=\calH\mid \Gamma_1, B^1\to^1 C^1 \tto \Delta_1\mid \Gamma_2 \tto B^2\to^2 C^2, B^3\to^3 C^3, \Delta_2$ (The superscript numbers are used to distinguish different occurrences of $A$, $B$ and $\to$).
    Assume that $B\to C$ does not occur in $\calH$, $\Gamma_1$, $\Delta_1$, $\Gamma_2$ or $\Delta_2$.
    We proceed with the construction in the order $B^1\to^1 C^1$, $B^2\to^2 C^2$ and $B^3\to^3 C^3$.
    By the construction, we will add nodes $\{\sigma1, \sigma11, \sigma12, \sigma111, \sigma112, \sigma121, \sigma122\}$ to $\DT_n(T, \calG)$ with labels
    \begin{align*}
         \calG^{\sigma1}&=\calG^{\sigma}\mid \Gamma_1 \tto \Delta_1\mid \Gamma_1, C^1 \tto B^1, \Delta_1,\\
         \calG^{\sigma11}&=\calG^{\sigma1}\mid \Gamma_2 \tto B^3\to^3 C^3, \Delta_2,\\
         \calG^{\sigma12}&=\calG^{\sigma1}\mid \Gamma_2, B^2 \tto C^2, B^3\to^3 C^3, \Delta_2,\\
         \calG^{\sigma111}&=\calG^{\sigma11}\mid \Gamma_2 \tto B^2\to^2 \Delta_2,\\
         \calG^{\sigma112}&=\calG^{\sigma11}\mid \Gamma_2, B^3 \tto C^3, B^2\to^2 C^2, \Delta_2,\\
         \calG^{\sigma121}&=\calG^{\sigma12}\mid \Gamma_2\tto B^2\to^2 C^2, \Delta_2,\\
         \calG^{\sigma122}&=\calG^{\sigma12}\mid \Gamma_2, B^3 \tto C^3, B^2\to^2 C^2, \Delta_2.
    \end{align*}
    The subtree of $\sigma$ is shown in Figure~{\ref{fig:construction for the formula appearing many times}}.
\end{example}

\figExampleOfRepeatedFormula


We now prove Theorem~\ref{thm:approximate strong completeness of GŁ∀} according to the following strategy.
First we show that if the construction of the $n$-th deduction tree halts, then we can obtain a desired derivation (Lemma~\ref{lem:finite tree implies provable}).
Then we show that if the construction of $\DT_n(T, \calG)$ does not halt, then there is a model of $T$ rejecting $\calG$ (Lemma~\ref{lem:infinite tree implies invalid}).
Theorem~\ref{thm:approximate strong completeness of GŁ∀} follows directly from the two lemmas above.

\subsection{Halting implies derivability}
\label{subsec:Halting implies derivability}

We will use following admissible approximate version of each logical rule in $\GL\forall$.

\begin{lemma}[\cite{wei2025}]\label{lem:1/n logical rules}
    The following rules are admissible in $\GL\forall$: 
    \begin{center}
        \scalebox{0.9}{
            \AxiomC{$\calG\mid \Gamma\tto_{\frac1n}\Delta\mid \Gamma, B\tto_{\frac1n} A, \Delta$}
            \RightLabel{\makebox[2em][l]{(${\to}\tto_{\frac1n}$)}}
            \UnaryInfC{$\calG\mid \Gamma, A\to B\tto_{\frac1n}\Delta$}
            \DisplayProof
            \qquad
            \AxiomC{$\calG\mid \Gamma\tto_{\frac1n}\Delta$}
            \AxiomC{$\calG\mid \Gamma, A\tto_{\frac1n} B, \Delta$}
            \RightLabel{\makebox[2em][l]{($\tto_{\frac1n}{\to}$)}}
            \BinaryInfC{$\calG\mid \Gamma\tto_{\frac1n} A\to B, \Delta$}
            \DisplayProof
        }\vspace{0.5em}
        \scalebox{0.9}{
            \AxiomC{$\calG\mid \Gamma, A(t)\tto_{\frac1n}\Delta$}
            \RightLabel{\makebox[2em][l]{($\forall\tto_{\frac1n}$)}}
            \UnaryInfC{$\calG\mid \Gamma, \forall xA(x)\tto_{\frac1n}\Delta$}
            \DisplayProof
            \qquad\qquad
            \AxiomC{$\calG\mid \Gamma\tto_{\frac1n} A(a), \Delta$}
            \RightLabel{\makebox[2em][l]{($\tto_{\frac1n}\forall$)}}
            \UnaryInfC{$\calG\mid \Gamma\tto_{\frac1n} \forall xA(x), \Delta$}
            \DisplayProof
        }\vspace{0.5em}
        \scalebox{0.9}{
            \AxiomC{$\calG\mid \Gamma, A(a)\tto_{\frac1n}\Delta$}
            \RightLabel{\makebox[2em][l]{($\exists\tto_{\frac1n}$)}}
            \UnaryInfC{$\calG\mid \Gamma, \exists xA(x)\tto_{\frac1n}\Delta$}
            \DisplayProof
            \qquad\qquad
            \AxiomC{$\calG\mid \Gamma\tto_{\frac1n} A(t), \Delta$}
            \RightLabel{\makebox[2em][l]{($\tto_{\frac1n}\exists$)}}
            \UnaryInfC{$\calG\mid \Gamma\tto_{\frac1n} \exists xA(x), \Delta$}
            \DisplayProof
        }
    \end{center}
    where $a$ is a free variable which does not occur in the conclusion of ($\exists\tto_{\frac1n}$) or ($\tto_{\frac1n}\forall$), $t$ is any $\calL$-term.
\end{lemma}

\begin{proof}
    Although the proof can be found in \cite[Lemma 4.4.7 and 4.4.8]{wei2025}, we prove the admissibility of ($\tto_{\frac1n}\forall$) here, as it helps clarify the mid-hypersequent theorem; see Remark~\ref{rmk:midhypersequent}.

    Assume $\vdash_{\GL\forall} \calG\mid \bot, n\Gamma\tto nA(a), n\Delta$, where $a$ is a free variable which does not occur in $\calG$, $\Gamma$ or $\Delta$.
    We can replace $a$ in the derivation with new fresh variables $a_1, \cdots, a_n$, then we get $\vdash_{\GL\forall} \calG\mid \bot, n\Gamma\tto nA(a_i), n\Delta$ for each $i$.
    Then we have the following derivation
    \begin{center}
        \AxiomC{$\calG\mid \bot, n\Gamma\tto nA(a_1), n\Delta$ \qquad $\cdots$ \qquad $\calG\mid \bot, n\Gamma\tto nA(a_n), n\Delta$}
        \RightLabel{(M)}
        \UnaryInfC{$\calG\mid n\bot, n^2\Gamma\tto nA(a_1), \cdots, nA(a_n), n^2\Delta$}
        \RightLabel{(S)}
        \UnaryInfC{$\calG\mid \bot, n\Gamma\tto A(a_1), \cdots, A(a_n), n\Delta\mid\cdots\mid \bot, n\Gamma\tto A(a_1), \cdots, A(a_n), n\Delta$}
        \RightLabel{(EC)}
        \UnaryInfC{$\calG\mid \bot, n\Gamma\tto A(a_1), \cdots, A(a_n), n\Delta$}
        \RightLabel{($\tto\forall$)}
        \UnaryInfC{$\calG\mid \bot, n\Gamma\tto n\forall x A(x), n\Delta$}
        \DisplayProof.
    \end{center}
    Hence $\vdash_{\GL\forall} \calG\mid \Gamma\tto_{\frac1n} \forall xA(x), \Delta$.
\end{proof}

The following lemma shows that the construction of the deduction tree corresponds to a proof search.

\begin{lemma}\label{lem:parent is derivable from children}
    Assume $\DT_n(T, \calG)$ is given.
    Let $\sigma$ be a non-leaf node and $\{\sigma _i\}_{i\in I}$ be the set of its child(ren).
    Then $\calG^{\sigma}_{\frac1n}$ is derivable from $\{\calG^{\sigma_i}_{\frac1n}\}_{i\in I}$ by using (EC) and rules in Lemma~\ref{lem:1/n logical rules}.
\end{lemma}

\begin{proof}
    The followings are the required derivations:
    \begin{center}
        \AxiomC{$\calH_{\frac1n} \mid \Gamma\tto_{\frac1n} B\to C, \Delta\mid \Gamma\tto_{\frac1n} \Delta$}
        \AxiomC{$\calH_{\frac1n} \mid \Gamma\tto_{\frac1n} B\to C, \Delta\mid \Gamma, B\tto_{\frac1n} C, \Delta$}
        \RightLabel{($\tto_{\frac1n}{\to}$)}
        \BinaryInfC{$\calH_{\frac1n}\mid \Gamma\tto_{\frac1n} B\to C, \Delta\mid \Gamma\tto_{\frac1n} B\to C, \Delta$}
        \RightLabel{(EC)}
        \UnaryInfC{$\calH_{\frac1n}\mid \Gamma\tto_{\frac1n} B\to C, \Delta$}
        \DisplayProof,
        \\\vspace{1em}

        \AxiomC{$\calH_{\frac1n}\mid \Gamma, B\to C\tto_{\frac1n} \Delta\mid \Gamma\tto_{\frac1n} \Delta\mid \Gamma, C\tto_{\frac1n} B, \Delta$}
        \RightLabel{(${\to}\tto_{\frac1n}$)}
        \UnaryInfC{$\calH_{\frac1n}\mid \Gamma, B\to C\tto_{\frac1n} \Delta\mid \Gamma, B\to C\tto_{\frac1n} \Delta$}
        \RightLabel{(EC)}
        \UnaryInfC{$\calH_{\frac1n}\mid \Gamma, B\to C\tto_{\frac1n} \Delta$}
        \DisplayProof, 
        \\\vspace{1em}

        \AxiomC{$\calH_{\frac1n}\mid \Gamma\tto_{\frac1n} \exists x B(x), \Delta\mid \Gamma\tto_{\frac1n}  B(t), \Delta$}
        \RightLabel{($\tto_{\frac1n}{\exists}$)}
        \UnaryInfC{$\calH_{\frac1n}\mid \Gamma\tto_{\frac1n} \exists x B(x), \Delta\mid \Gamma\tto_{\frac1n} \exists x B(x), \Delta$}
        \RightLabel{(EC)}
        \UnaryInfC{$\calH_{\frac1n}\mid \Gamma\tto_{\frac1n} \exists x B(x), \Delta$}
        \DisplayProof, 
        \\\vspace{1em}

        \AxiomC{$\calH_{\frac1n}\mid \Gamma, \exists x B(x)\tto_{\frac1n}  \Delta\mid \Gamma, B(a)\tto_{\frac1n} \Delta$}
        \RightLabel{(${\exists}\tto_{\frac1n}$)}
        \UnaryInfC{$\calH_{\frac1n}\mid \Gamma, \exists x B(x)\tto_{\frac1n}  \Delta\mid \Gamma, \exists x B(x)\tto_{\frac1n} \Delta$}
        \RightLabel{(EC)}
        \UnaryInfC{$\calH_{\frac1n}\mid \Gamma, \exists x B(x)\tto_{\frac1n}  \Delta$}
        \DisplayProof, 
        \\\vspace{1em}

        \AxiomC{$\calH_{\frac1n}\mid \Gamma\tto_{\frac1n} \forall x B(x), \Delta\mid \Gamma\tto_{\frac1n}  B(a), \Delta$}
        \RightLabel{($\tto_{\frac1n}{\forall}$)}
        \UnaryInfC{$\calH_{\frac1n}\mid \Gamma\tto_{\frac1n} \forall x B(x), \Delta\mid \Gamma\tto_{\frac1n} \forall x B(x), \Delta$}
        \RightLabel{(EC)}
        \UnaryInfC{$\calH_{\frac1n}\mid \Gamma\tto_{\frac1n} \forall x B(x), \Delta$}
        \DisplayProof, 
        \\\vspace{1em}

        \AxiomC{$\calH_{\frac1n}\mid \Gamma, \forall x B(x)\tto_{\frac1n}  \Delta\mid \Gamma, B(t)\tto_{\frac1n} \Delta$}
        \RightLabel{(${\forall}\tto_{\frac1n}$)}
        \UnaryInfC{$\calH_{\frac1n}\mid \Gamma, \forall x B(x)\tto_{\frac1n}  \Delta\mid \Gamma, \forall x B(x)\tto_{\frac1n} \Delta$}
        \RightLabel{(EC)}
        \UnaryInfC{$\calH_{\frac1n}\mid \Gamma, \forall x B(x)\tto_{\frac1n}  \Delta$}
        \DisplayProof.
    \end{center}
    \vspace{-1.5em}
\end{proof}

We can now prove the main lemma of this subsection.

\begin{lemma}\label{lem:finite tree implies provable}
    If the construction of $\DT_n(T, \calG)$ halts, then there exist $m_1, \cdots, m_k\in\nn_{>0}$ and $C_1, \cdots, C_k\in T$ such that
    \begin{equation*}
        \vdash_{\GL\forall} m_1C_1\tto_{\frac1n}\bot\mid\cdots \mid m_kC_k\tto_{\frac1n}\bot \mid \calG_{\frac1n}.
    \end{equation*}
\end{lemma}

\begin{proof}
    If the construction of $\DT_n(T, \calG)$ halts, then $\DT_n(T, \calG)$ is finite and each leaf of $\DT_n(T, \calG)$ is non-active.
    We will show that $\calG^\sigma_{\frac1n}$ is derivable in $\GL\forall$ for each $\sigma\in \DT_n(T, \calG)$ by induction on the height of $\sigma$.

    If $\sigma$ is a leaf, then $\prop(\calG^\sigma)_{\frac1n}$ is derivable in $\GL$ by the non-activeness of $\sigma$.
    Since $\prop(\calG^\sigma)\subseteq \calG^\sigma$, $\calG^\sigma_{\frac1n}$ is derivable from $\prop(\calG^\sigma)_{\frac1n}$ by (EW).
    Hence $\calG^\sigma_{\frac1n}$ is derivable in $\GL\forall$.

    If $\sigma$ is a non-leaf node with children $\{\sigma_i\}_{i\in I}$, then by induction hypothesis, $\{\calG^{\sigma_i}_{\frac1n}\}_{i\in I}$ are derivable in $\GL\forall$.
    Lemma~\ref{lem:parent is derivable from children} shows that $\calG^\sigma_{\frac1n}$ is derivable from $\{\calG^{\sigma_i}_{\frac1n}\}_{i\in I}$ by using (EC) and rules in Lemma~\ref{lem:1/n logical rules}.
    By the admissibility of these rules, $\calG^\sigma_{\frac1n}$ is also derivable in $\GL\forall$.
    This completes the induction.
    In particular, $\calG^{\varepsilon}_{\frac1n}$ is derivable in $\GL\forall$.
    We finish the proof by noting that $\calG^{\varepsilon}$ is of the form
    \begin{equation*}
        m_1C_1\tto\bot\mid\cdots \mid m_kC_k\tto\bot \mid \calG
    \end{equation*}
    where $C_1, \cdots, C_k\in T$ and $m_1, \cdots, m_k\in\nn_{>0}$.
\end{proof}

Before continuing with the proof of Theorem~\ref{thm:approximate strong completeness of GŁ∀}, we pause here to verify Lemma~\ref{lem:finite tree implies provable} with a well-known example---drinker's paradox; see also \cite{baaz2010, metcalfe2008, wei2025}.

\figDeductionTreeOfDrinkerParadox

\begin{example}[Drinker's paradox]\label{eg:drinker's paradox}
    Assume that language $\calL$ contains a unary predicate $P$.
    We will consider the $\calL$-formula $\exists x \forall y (P(x)\to P(y))$, which is $[0, 1]$-valid, yet is not derivable in $\GL\forall$; see \cite{baaz2010, metcalfe2008}.
    We now construct a deduction tree $\DT_n(\varnothing, {}\tto \exists x \forall y (P(x)\to P(y)))$ and recover a $\GL\forall$-derivation of ${}\tto_{\frac1n} \exists x \forall y (P(x)\to P(y))$; while an alternative approach to find such a derivation can be found in \cite{wei2025}.
    We only present the case $n=2$; other cases are easily generalized.
    We fix an enumeration $\{(A_i, t_i)\}_{i=1}^{+\infty}$ satisfying
    \begin{align*}
        A_1 &= \exists x \forall y (P(x)\to P(y)), \quad t_1=a, \\
        A_2 &= \forall y (P(a)\to P(y)), \\
        A_3 &= \exists x \forall y (P(x)\to P(y)), \quad t_3=b, \\
        A_4 &= \forall y (P(b)\to P(y)).
    \end{align*}
    The deduction tree $\DT_2(\varnothing, {}\tto \exists x\forall y(P(x)\to P(y)))$ after first four stages of its construction, shown in Figure~\ref{fig:example of 2-deduction tree of drinker's paradox}, has five nodes $\varepsilon, 1, 11, 111, 1111$ with labels
    \begin{align*}
        \calG^{\varepsilon}&= {{} \tto \exists x \forall y (P(x)\to P(y))}, \\
        \calG^{1}&=\calG^{\varepsilon} \mid {{} \tto \forall y(P(a)\to P(y))}, \\
        \calG^{11}&=\calG^1 \mid {{} \tto P(a)\to P(b)}, \\
        \calG^{111}&=\calG^{11} \mid {{} \tto \forall y(P(b)\to P(y))}, \\
        \calG^{1111}&=\calG^{111} \mid {{} \tto P(b)\to P(c)}.
    \end{align*}
    Fresh free variables used in stage 2 and 4 are $b$ and $c$ respectively.
    One can show that
    \begin{equation*}
        \prop(\calG^{1111})_{\frac12}= {} \tto_{\frac12} P(a)\to P(b)\mid {} \tto_{\frac12} P(b)\to P(c)
    \end{equation*}
    is $[0, 1]$-valid.
    Hence the only leaf $1111$ of $\DT_2(\varnothing, {}\tto \exists x \forall y (P(x)\to P(y)))$ is nonactive, and construction halts here.
    According to the proof of Lemma~\ref{lem:finite tree implies provable}, we can get a $\GL\forall$-derivation of $\tto_{\frac12} \exists x\forall y(P(x)\to P(y))$ along the deduction tree bottom-up, which is shown in Figure~\ref{fig:strict derivation}.
    Another simplified derivation, obtained by deleting unnecessary applications of (EW) and (EC), is shown in Figure~\ref{fig:simplified derivation}.
\end{example}

\figOriginalDerivationOfDrinkerParadox

\figSimplifiedDerivationOfDrinkerParadox

\subsection{Non-halting implies existence of a rejecting model}
\label{subsec:Non-halting implies existence of a rejecting model}

We resume the proof of Theorem~\ref{thm:approximate strong completeness of GŁ∀}.
The rejecting model constructed later is the so-called ``term interpretation''.

\begin{definition}\label{def:term interpretation}
    Denote by $X_{\calL}$ the set of all atomic $\calL$-formulas.
    For each $[0, 1]$-evaluation $v\in [0, 1]^{X_{\calL}}$, we define a $[0, 1]$-interpretation $\frakI_v=(\calM_v, \beta)$ as follows: 
    \begin{enumerate-abc}
        \item The universe $M$ of $\calM_v$ is the set of $\calL$-terms.
        \item For each constant symbol $e$, $e^{\calM_v}$ is the $\calL$-term $e$ itself.
        \item For each $k$-ary function symbol $f$, $f^{\calM_v}(t_1, \cdots, t_k)$ is the $\calL$-term $f(t_1, \cdots, t_k)$.
        \item For each $k$-ary predicate symbol $P$, $P^{\calM_v}(t_1, \cdots, t_k)=v(P(t_1, \cdots, t_k))$.
        \item For each free variable $a$, $\beta(a)$ is $a$ itself.
        For each bound variable $x$, $\beta(x)$ can be chosen arbitrarily.
    \end{enumerate-abc}
    Such a $[0, 1]$-interpretation is called a \emph{term interpretation}.
\end{definition}

\begin{remark}\label{rmk:term interpretation}
    Let $\frakI_v=(\calM_v, \beta)$ be a term interpretation.
    Note that $\beta(t)=t$ for each $\calL$-term $t$.
    Hence 
    \begin{equation*}
        \fs{\exists xA(x)}_{\frakI_v}
        =\inf_{t\in M} \fs{A(x)}_{\frakI_v[x\mapsto t]}
        =\inf_{t\in M} \fs{A(x)}_{\frakI_v[x\mapsto \beta(t)]}
        =\inf_{t\in M} \fs{A(t)}_{\frakI_v}
    \end{equation*}
    for each $\calL$-formula of the form $\exists xA(x)$ by substitution lemma (Proposition~\ref{prop:substitution lemma}).
    Similarly, $\fs{\forall xA(x)}_{\frakI_v}=\sup_{t\in M} \fs{A(t)}_{\frakI_v}$ for each $\calL$-formula of the form $\forall xA(x)$.
\end{remark}

To apply induction, we define the \emph{complexity} of $\calL$-semihypersequents inductively as follows: 
\begin{enumerate-abc}
    \item $\cp(\bot)=\cp(A)=0$, where $A$ is atomic.
    \item $\cp(A\to B)=\cp(A)+\cp(B)+1$.
    \item $\cp(\exists xA)=\cp(\forall xA)=\cp(A)+1$.
    \item $\cp(\Gamma)=\sum_{A\in\Gamma} \cp(A)$.
    \item $\cp(\Gamma_1\tto\Delta_1\mid\cdots\mid \Gamma_k\tto \Delta_k)=\sum_{i=1}^k (\cp(\Gamma_i)+\cp(\Delta_i))$.
\end{enumerate-abc}
Actually, the complexity of $\calG$ is the number of logical connectives and quantifiers it contains.

The following is a technical lemma at the propositional level.
It shows that if every finite sub-hypersequent admits a rejecting evaluation with a uniform positive lower bound, so does the original infinite hypersequent.
The proof uses a compactness argument; see \cite{munkres2000} for the topological notions involved.

\begin{lemma}\label{lem:propositional hypersequent with lower bound}
    Let $\calH$ be a propositional $\calL$-hypersequent (possibly infinite) and $\epsilon>0$.
    If for each finite $\calL$-sub-hypersequent $\calH'$ of $\calH$, $\tilde{w}(\calH')\geq\epsilon$ holds for some $w\in [0, 1]^{X_\calL}$, then there exists a $v\in [0, 1]^{X_\calL}$ such that $\tilde{v}(\calH)\geq\epsilon$.
\end{lemma}

\begin{proof}
    For each propositional $\calL$-sequent $\calS$, we define a function $f_{\calS}: [0, 1]^{X_\calL}\to\rr$, $v\mapsto \tilde{v}(\calS)$.
    We claim that $f_{\calS}$ is continuous for each propositional $\calL$-sequent $\calS=\Gamma\tto\Delta$, where $[0, 1]^{X_{\calL}}$ is equipped with the product topology induced by the standard topology on $[0, 1]$.
    We show the claim by induction on $\cp(\calS)$.

    If $\cp(\calS)=0$, i.e., $\calS$ is atomic, then we show the claim by induction on $|\Gamma|+|\Delta|$.
    If $\calS$ is the sequent ${}\tto {}$, then $f_{\calS}$ is the constant function $0$, hence continuous.
    Note that 
    \begin{alignat*}{3}
        f_{\Gamma\tto\Delta, \bot}&=f_{\Gamma\tto\Delta}+1, \qquad
        f_{\Gamma\tto\Delta, A}=f_{\Gamma\tto\Delta}+\pi_A, \\
        f_{\bot, \Gamma\tto\Delta}&=f_{\Gamma\tto\Delta}-1, \qquad
        f_{A, \Gamma\tto\Delta}=f_{\Gamma\tto\Delta}-\pi_A, 
    \end{alignat*}
    where $A\in X_{\calL}$ and $\pi_A: [0, 1]^{X_{\calL}}\to [0, 1]$, $\pi_A(v)=v(A)$ is the projection.
    It is well known that projections are continuous with respect to the product topology.
    Hence all of $f_{\Gamma\tto\Delta, \bot}$, $f_{\bot, \Gamma\tto\Delta}$, $f_{\Gamma\tto\Delta, A}$ and $f_{A, \Gamma\tto\Delta}$ are continuous by induction hypothesis.
    This completes the induction.

    We note that 
    \begin{align*}
        f_{\Gamma\tto A\to B, \Delta}&=\max\{f_{\Gamma\tto\Delta}, f_{\Gamma, A\tto B, \Delta}\}, \\
        f_{\Gamma, A\to B\tto\Delta}&=\min\{f_{\Gamma\tto\Delta}, f_{\Gamma, B\tto A, \Delta}\}, 
    \end{align*}
    hence the continuity of $f_{\Gamma\tto\Delta}$ and $f_{\Gamma, A\tto B, \Delta}$ implies the continuity of $f_{\Gamma\tto A\to B, \Delta}$, and the continuity of $f_{\Gamma\tto\Delta}$ and $f_{\Gamma, B\tto A, \Delta}$ implies the continuity of $f_{\Gamma, A\to B\tto\Delta}$.
    The first induction goes through, and the claim follows.

    By the above claim we know that $\{f_{\calS}^{-1}([\epsilon, +\infty))\}_{\calS\in \calH}$ is a family of closed subsets of $[0, 1]^{X_{\calL}}$.
    The assumption given in this lemma implies that this family satisfies the finite intersection property.
    Hence $\bigcap_{\calS\in \calH} f_{\calS}^{-1}([\epsilon, +\infty))$ is nonempty by the compactness of $[0, 1]^{X_{\calL}}$.
    Let $v\in \bigcap_{\calS\in \calH} f_{\calS}^{-1}([\epsilon, +\infty))$, then $\tilde{v}(\calH)=\inf_{\calS\in \calH} \tilde{v}(\calS)=\inf_{\calS\in \calH} f_\calS(v)\geq\epsilon$, as required.
\end{proof}

We are now in a position to prove the existence of a rejecting model.

\begin{lemma}\label{lem:infinite tree implies invalid}
    If the construction of $\DT_n(T, \calG)$ does not halt, then there is a $[0, 1]$-model $\frakI$ of $T$ such that $\fs{\calG}_{\frakI}\geq\frac1n$.
\end{lemma}

\begin{proof}
    As mentioned in Remark~\ref{rmk: construction of deduction tree}~(1), there are three cases.
    If $T=\varnothing$ and $\calG$ is atomic, then $\calG_{\frac1n}=\prop(\calG)_{\frac1n}$ is not $[0, 1]$-valid, which means that there is a $v\in [0, 1]^{X_\calL}$ such that $\tilde{v}(\calG_{\frac1n})>0$, i.e., $\tilde{v}(\calG)>\frac1n$.
    Hence $\fs{\calG}_{\frakI_v}\geq\frac1n$, as desired.

    If $T$ is a non-empty set of atomic $\calL$-formulas and $\calG$ is atomic, let $\calG^{\infty}=\calG^{\varepsilon}$.
    If $\DT_n(T, \calG)$ is infinite, since $\DT_n(T, \calG)$ is a binary tree, there is an infinite path $(\sigma_0=\varepsilon, \sigma_1, \sigma_2, \cdots)$ by weak König's lemma.
    Since $\calG^{\sigma_0}\subseteq\calG^{\sigma_1}\subseteq\calG^{\sigma_2}\subseteq\cdots$, we define $\calG^{\infty}=\bigsqcup_{i=0}^{+\infty} \calG^{\sigma_i}$ be the limit.
    Note that $\calG^{\infty}$ is an infinite $\calL$-hypersequent in both cases.

    By the condition of activeness, for each finite $\calL$-sub-hypersequent $\calH$ of $\prop(\calG^{\infty})$, $\calH_{\frac1n}$ is not $[0, 1]$-valid, that is, there exists a $w\in [0, 1]^{X_\calL}$ such that $\tilde{w}(\calH_{\frac1n})>0$, i.e., $\tilde{w}(\calH)>\frac1n$.
    By Lemma~\ref{lem:propositional hypersequent with lower bound}, there exists a $v\in [0, 1]^{X_\calL}$ such that $\tilde{v}(\prop(\calG^{\infty}))\geq\frac1n$.
    We will show that $\frakI_v$ is a required $[0, 1]$-interpretation.
    We write $\fs{~}_{\frakI_v}$ as $\fs{~}$ for simplicity.

    {\bfseries Claim.} $\fs{\calS}\geq\frac1n$ for each $\calS\in\calG^{\infty}$.

    We show the claim by induction on the complexity of $\calS$.

    If $\calS$ is propositional, then $\calS\in\prop(\calG^{\infty})$, hence $\fs{\calS}=\tilde{v}(\calS)\geq \tilde{v}(\prop(\calG^{\infty})) \geq\frac1n$.

    If $\calS$ is of the form $\Gamma\tto B\to C, \Delta$, then $\Gamma\tto \Delta\in\calG^{\infty}$ or $\Gamma, B \tto C, \Delta\in\calG^{\infty}$.
    By induction hypothesis, we have 
    \begin{equation*}
        \fs{\Gamma\tto \Delta}\geq\frac1n
        \quad \text{or} \quad
        \fs{\Gamma, B\tto C, \Delta}\geq\frac1n.
    \end{equation*}
    Hence
    \begin{align*}
        \fs{\Gamma\tto B\to C, \Delta}
        &=\fs{\Gamma\tto \Delta}+\max\{\fs{C}-\fs{B}, 0\}\\
        &=\max\{\fs{\Gamma, B\tto C, \Delta}, \fs{\Gamma\tto \Delta}\}\geq\frac1n.
    \end{align*}

    If $\calS$ is of the form $\Gamma, B\to C\tto \Delta$, then $\Gamma\tto \Delta\in\calG^{\infty}$ and $\Gamma, C \tto B, \Delta\in\calG^{\infty}$.
    By induction hypothesis, we have 
    \begin{equation*}
        \fs{\Gamma\tto \Delta}\geq\frac1n
        \quad \text{and} \quad
        \fs{\Gamma, C\tto B, \Delta}\geq\frac1n.
    \end{equation*}
    Hence
    \begin{align*}
        \fs{\Gamma, B\to C\tto \Delta}
        &=\fs{\Gamma\tto \Delta}-\max\{\fs{C}-\fs{B}, 0\}\\
        &=\min\{\fs{\Gamma, C\tto B, \Delta}, \fs{\Gamma\tto \Delta}\}\geq\frac1n.
    \end{align*}

    If $\calS$ is of the form $\Gamma\tto \exists x B(x), \Delta$, then $\Gamma\tto B(t), \Delta\in\calG^{\infty}$ for each $\calL$-term $t$ by our assumption on the enumeration $\{(A_i, t_i)\}_{i=1}^{+\infty}$.
    By induction hypothesis, we have 
    \begin{equation*}
        \fs{\Gamma\tto B(t), \Delta}\geq\frac1n
    \end{equation*}
    for each $\calL$-term $t$.
    Hence
    \begin{align*}
        \fs{\Gamma\tto \exists x B(x), \Delta}
        &= \fs{\Gamma\tto \Delta}+\inf_{t\in M}\fs{B(t)}\\
        &= \inf_{t\in M}\fs{\Gamma\tto B(t), \Delta}\geq\frac1n
    \end{align*}
    by Remark~\ref{rmk:term interpretation}.

    If $\calS$ is of the form $\Gamma, \exists x B(x)\tto \Delta$, then $\Gamma, B(a)\tto \Delta\in\calG^{\infty}$ for some free variable $a$.
    By induction hypothesis, we have 
    \begin{equation*}
        \fs{\Gamma, B(a)\tto \Delta}\geq\frac1n.
    \end{equation*}
    Hence
    \begin{align*}
        \fs{\Gamma, \exists x B(x)\tto \Delta}
        &= \fs{\Gamma\tto \Delta}-\inf_{t\in M}\fs{B(t)}\\
        &= \sup_{t\in M}\fs{\Gamma, B(t)\tto \Delta}\geq \fs{\Gamma, B(a)\tto \Delta} \geq\frac1n
    \end{align*}
    by Remark~\ref{rmk:term interpretation} again.

    If $\calS$ is of the form $\Gamma\tto \forall x B(x), \Delta$ or $\Gamma, \forall x B(x)\tto \Delta$, we can deal with them similarly to the previous two cases.
    This completes the induction and the claim follows.

    Now we show that $\frakI_v$ is a $[0, 1]$-model of $T$.
    For each $C\in T$, there exists an infinite increasing sequence $\{m_k\}_{k=1}^{+\infty}$ of positive natural numbers such that $m_k C\tto \bot\in \calG^{\infty}$ for each $k\in\nn_{>0}$ by the construction of $\DT_n(T, \calG)$.
    Hence by the claim, for each $k\in\nn_{>0}$, we have 
    \begin{equation*}
        1-m_k\fs{C}=\fs{m_kC\tto \bot}\geq\frac1n, 
    \end{equation*}
    i.e., $\fs{C}\leq \frac{1}{m_k}\cdot\frac{n-1}{n}$.
    Taking $k \to +\infty$ then yields $\fs{C}=0$.

    Finally, note that $\calG\subseteq\calG^{\infty}$, hence 
    \begin{equation*}
        \fs{\calG}=\min_{\calS\in \calG} \fs{\calS}\geq\frac1n
    \end{equation*}
    by the claim.
\end{proof}

In view of the work done above, we can now easily prove our main theorem.

\begin{proof}[Proof of Theorem~\ref{thm:approximate strong completeness of GŁ∀}]
    If the construction of $\DT_n(T, \calG)$ does not halt for some $n\in\nn_{>0}$, then $\fs{\calG}_{\frakI}\geq\frac1n>0$ for some $[0, 1]$-model $\frakI$ of $T$ by Lemma~\ref{lem:infinite tree implies invalid}, which contradicts to $T\models_{[0, 1]} \calG$.
    Hence the construction of $\DT_n(T, \calG)$ halts for all $n\in\nn_{>0}$.
    The remaining follows from Lemma~\ref{lem:finite tree implies provable}.
\end{proof}

\subsection{A sequent-level cut rule}
\label{subsec:scut}

At the end of this section, we discuss cut rules in $\GL\forall$.
There is already a formula-level cut rule
\begin{center}
    \scalebox{0.9}{
    \AxiomC{$\calG \mid \Gamma, A\tto A, \Delta$}
    \RightLabel{(Cut)}
    \UnaryInfC{$\calG \mid \Gamma \tto \Delta$}
    \DisplayProof
    }
    ~, 
\end{center}
and a completeness result \cite[Theorem 9]{baaz2010}: $\vdash_{\GL\forall+(\text{Cut})} {} \tto A$ if and only if $A$ is valid under the general semantics.

Since the contraction rule (EC) is sequent-level rather than formula-level, a sequent-level cut rule sequents is also of interest.
We define such a cut rule (s-Cut): 
\begin{center}
    \scalebox{0.9}{
    \AxiomC{$\calG \mid \Gamma \tto \Delta$}
    \AxiomC{$\calG \mid \Delta \tto \Gamma, \bot$}
    \RightLabel{(s-Cut)}
    \BinaryInfC{$\calG$}
    \DisplayProof
    }
    (provided $\calG\neq\varnothing$).
\end{center}
The following theorem illustrates the naturalness of the new cut rule.
A similar result in classical first-order logic can be found in \cite[Chapter 1, 2.3.7]{buss1998}.

\begin{theorem}\label{thm:approximate strong completeness of GŁ∀+(s-Cut)}
    Let $\calL$ be a countable language, $T$ be an $\calL$-theory and $\calG$ be an $\calL$-hypersequent.
    The following statements are equivalent:
    \begin{enumerate-abc}
        \item $T\models_{[0, 1]} \calG$.
        \item For each $n\in\nn_{>0}$, there exist $m_1, \cdots, m_k\in\nn_{>0}$ and $C_1, \cdots, C_k\in T$ such that
        \begin{equation*}
            \vdash_{\GL\forall} m_1C_1\tto_{\frac1n}\bot\mid\cdots \mid m_kC_k\tto_{\frac1n}\bot \mid \calG_{\frac1n}.
        \end{equation*}
        \item $T \vdash_{\GL\forall+\text{(s-Cut)}} \calG_{\frac1n}$ for each $n\in\nn_{>0}$.
    \end{enumerate-abc}
\end{theorem}

\begin{proof}
    (a) $\tto$ (b): This is exactly Theorem~\ref{thm:approximate strong completeness of GŁ∀}.

    (b) $\tto$ (c): For $n\geq 2$, there exist $m_1, \cdots, m_k\in\nn_{>0}$ and $C_1, \cdots, C_k\in T$ such that $\vdash_{\GL\forall} m_1C_1\tto_{\frac1n}\bot\mid\cdots \mid m_kC_k\tto_{\frac1n}\bot \mid \calG_{\frac1n}$.
    The case of $k=0$ is trivial.
    Assume $k>0$ and $\calH = m_2C_2\tto_{\frac1n}\bot\mid\cdots \mid m_kC_k\tto_{\frac1n}\bot \mid \calG_{\frac1n}$.
    We have the following derivation:
    \\
    \begin{center}
        \AxiomC{$\calH\mid \bot, nm_1C_1\tto n\bot$}
        \AxiomC{$ $}
        \RightLabel{($C_1\in T$)}
        \UnaryInfC{$\tto C_1$}
        \RightLabel{(M)}
        \UnaryInfC{$\tto nm_1C_1$}
        \AxiomC{$ $}
        \RightLabel{(ID)}
        \UnaryInfC{$\bot \tto \bot$}
        \RightLabel{(M)}
        \BinaryInfC{$\bot \tto \bot, nm_1C_1$}
        \RightLabel{(EW) and (IW)}
        \UnaryInfC{$\calH\mid (n-1)\bot \tto \bot, nm_1C_1$}
        \RightLabel{(s-Cut)}
        \BinaryInfC{$\calH$}
        \DisplayProof~.
    \end{center}
    Then we can obtain a $\GL\forall+\text{(s-Cut)}$-derivation of $\calG_{\frac1n}$ from $T$ by iterating this process.

    If $n=1$, assume $\calG=\Gamma_1\tto \Delta_1\mid \cdots\mid \Gamma_k\tto \Delta_k$.
    From the previous case, we have $T \vdash_{\GL\forall+\text{(s-Cut)}} \calG_{\frac12}$.
    We note that $\calG_{\frac11}$ is derivable from $\calG_{\frac12}$ by the following derivation:
    \begin{center}
        \AxiomC{$\bot, 2 \Gamma_1\tto 2\Delta_1\mid \cdots\mid \bot, 2 \Gamma_k\tto 2\Delta_k$}
        \RightLabel{(IW)}
        \UnaryInfC{$2\bot, 2 \Gamma_1\tto 2\Delta_1\mid \cdots\mid 2\bot, 2 \Gamma_k\tto 2\Delta_k$}
        \RightLabel{(S)}
        \UnaryInfC{$\bot, \Gamma_1\tto \Delta_1\mid \bot, \Gamma_1\tto \Delta_1\mid \cdots\mid \bot, \Gamma_k\tto \Delta_k\mid \bot, \Gamma_k\tto \Delta_k$}
        \RightLabel{(EC)}
        \UnaryInfC{$\bot, \Gamma_1\tto \Delta_1\mid \cdots\mid \bot, \Gamma_k\tto \Delta_k$}
        \DisplayProof , 
    \end{center}
    hence $T \vdash_{\GL\forall+\text{(s-Cut)}} \calG_{\frac11}$.

    (c) $\tto$ (a): We can show the strong soundness of $\GL\forall+\text{(s-Cut)}$: If $T\vdash_{\GL\forall+\text{(s-Cut)}} \calG$, then $T\models_{[0, 1]} \calG$.
    The proof can be carried out by the same method as that of Theorem~\ref{thm:strong soundness}.
    For this purpose, it only remains to establish that if the premises of (s-Cut) are true in every $[0, 1]$-model of $T$, so is the conclusion of (s-Cut).
    Let $\frakI$ be a $[0, 1]$-model of $T$ such that $\fs{\calG \mid \Gamma \tto \Delta}_{\frakI}\leq 0$ and $\fs{\calG \mid \Delta \tto \Gamma, \bot}_{\frakI}\leq 0$.
    If $\fs{\calG}_{\frakI}> 0$, then 
    \begin{align*}
        \fs{\Delta}_{\frakI}-\fs{\Gamma}_{\frakI}&=\fs{\Gamma \tto \Delta}_{\frakI}\leq 0, \\
        1+\fs{\Gamma}_{\frakI}-\fs{\Delta}_{\frakI}&=\fs{\Delta \tto \Gamma, \bot}_{\frakI}\leq 0.
    \end{align*}
    Hence $\fs{\Gamma}_{\frakI}\geq \fs{\Delta}_{\frakI}\geq 1+\fs{\Gamma}_{\frakI}$, a contraction.
    Then we have $\fs{\calG}_{\frakI}\leq 0$.

    For each $n\in\nn_{>0}$, since $T \vdash_{\GL\forall+\text{(s-Cut)}} \calG_{\frac1n}$, we have $T\models_{[0, 1]} \calG_{\frac1n}$ by above strong soundness.
    That is to say, for each $[0, 1]$-model $\frakI$ of $T$ and $n\in\nn_{>0}$, we have $\fs{\calG_{\frac1n}}_{\frakI}\leq 0$, i.e., $\fs{\calG}_{\frakI}\leq\frac1n$.
    Hence $\fs{\calG}_{\frakI}\leq 0$.
    This completes the proof.
\end{proof}

\section{Applications}
\label{sec:applications}
In this section, we discuss several applications of our main theorem.
Needless to say, the approximate completeness of $\GL\forall$ for all hypersequents (Theorem~\ref{thm:approximate completeness of GŁ∀}) follows directly by taking $T=\varnothing$.

We discuss compactness first.
It is well known that $[0, 1]$-consequence in Łukasiewicz logic is not compact; see \cite[Remark 3.2.14]{hajek2001} for a counterexample.
However, the following result shows that approximate $[0, 1]$-consequence is compact.

\begin{theorem}\label{thm:compactness of approximate logical entailment}
    Let $\calL$ be a countable language, $T$ be an $\calL$-theory and $\calG$ be an $\calL$-hypersequent.
    The following statements are equivalent:
    \begin{enumerate-abc}
        \item $T\models_{[0, 1]} \calG$.
        \item For each $n\in\nn_{>0}$, there is a finite subset $T_n\subseteq T$ and $m\in\nn_{>0}$ such that for each $[0, 1]$-interpretation $\frakI$, if $\max_{C\in T_n} \fs{C}_{\frakI}\leq\frac{1}{m}$, then $\fs{\calG}_{\frakI}\leq\frac{1}{n}$.
        \item For each $n\in\nn_{>0}$, there is a finite subset $T_n\subseteq T$ such that $T_n\models_{[0, 1]} \calG_{\frac1n}$.
        \item For each $n\in\nn_{>0}$, $T\models_{[0, 1]} \calG_{\frac1n}$.
    \end{enumerate-abc}
\end{theorem}

\begin{proof}
    (a) $\tto$ (b): We fix $n\in\nn_{>0}$.
    By Theorem~\ref{thm:approximate strong completeness of GŁ∀}, there are $C_1,\cdots,C_k\in T$ and $m_1,\cdots m_k\in \nn_{>0}$ such that 
    \begin{equation*}
        \vdash_{\GL\forall} m_1 C_1 \tto_{\frac{1}{n+1}}\bot\mid\cdots \mid m_k C_k\tto_{\frac{1}{n+1}}\bot \mid \calG_{\frac{1}{n+1}}.
    \end{equation*}
    Now we show that finite subset $T_n=\{C_1,\cdots, C_k\}$ and positive integer $m=1+2\max_{1\leq i\leq k}{m_i}$ are the desired objects.

    Let $\frakI=(\calM, \beta)$ be a $[0, 1]$-interpretation satisfying $\max_{1\leq i\leq k}\fs{C_i}_{\frakI}\leq \frac{1}{m}$.
    Then for each $1\leq i\leq k$, we have
    \begin{align*}
        \fs{m_i C_i\tto_{\frac{1}{n+1}}\bot}_{\frakI}
        &=(n+1)(1-m_i\fs{ C_i}_{\frakI})-1\\
        &\geq (1+1)\left(1-\frac{m-1}{2}\cdot\frac{1}{m}\right)-1
        =\frac{1}{m}>0.
    \end{align*}
    Since $m_1 C_1 \tto_{\frac{1}{n+1}}\bot\mid\cdots \mid m_k C_k\tto_{\frac{1}{n+1}}\bot \mid \calG_{\frac{1}{n+1}}$ is derivable in $\GL\forall$, we have 
    \begin{equation*}
        \min \left\{\fs{m_1 C_1\tto_{\frac{1}{n+1}}\bot}_{\frakI}, \cdots, \fs{m_k C_k\tto_{\frac{1}{n+1}}\bot}_{\frakI}, \fs{\calG_{\frac{1}{n+1}}}_{\frakI} \right\}\leq 0
    \end{equation*}
    by the soundness of $\GL\forall$.
    Hence $\fs{\calG_{\frac{1}{n+1}}}_{\frakI}\leq 0$, i.e., $\fs{\calG}_{\frakI}\leq \frac{1}{n+1}\leq\frac{1}{n}$.

    (b) $\tto$ (c): If $\frakI$ is  a $[0, 1]$-model of $T_n$, then $\max_{C\in T_n} \fs{C}_{\frakI}=0\leq\frac1m$.

    (c) $\tto$ (d): Every $[0,1]$-model of $T$ is also a $[0,1]$-model of $T_n$.

    (d) $\tto$ (a): Let $\frakI=(\calM, \beta)$ be a $[0, 1]$-model of $T$.
    For each $n\in\nn_{>0}$, since $T\models_{[0, 1]} \calG_{\frac1n}$, we have $\fs{\calG_{\frac1n}}_{\frakI}\leq 0$, i.e., $\fs{\calG}_{\frakI}\leq \frac1n$.
    Hence $\fs{\calG}_{\frakI}\leq 0$.
\end{proof}

Then we prove a variant of Gentzen's mid-sequent theorem.
We say an $\calL$-semihypersequent $\calG$ (possibly infinite) is \emph{prenex} if each formula in $\calG$ is of the prenex form.

\begin{theorem}[Mid-hypersequent theorem]\label{thm:midhypersequent}
    Let $\calL$ be a countable language, $T$ be an $\calL$-theory and each formula in $T$ is of the prenex form.
    Let $\calG$ be a prenex $\calL$-hypersequent.
    If $T\models_{[0, 1]}\calG$, then for each $n\in\nn_{>0}$, there exists a propositional $\calL$-hypersequent $\calG^{n}$ (called a mid-hypersequent) satisfying the following conditions: 
    \begin{enumerate-abc}
        \item $\vdash_{\GL} \calG^{n}$, and
        \item there exist $m_1, \cdots, m_k\in\nn_{>0}$ and $C_1, \cdots, C_k\in T$ such that $m_1C_1\tto_{\frac1n}\bot\mid\cdots \mid m_kC_k\tto_{\frac1n}\bot \mid \calG_{\frac1n}$ is derivable from $\calG^{n}$ by using ($\forall\tto_{\frac1n}$), ($\tto_{\frac1n}\forall$), ($\exists\tto_{\frac1n}$), ($\tto_{\frac1n}\exists$), (EW) and (EC).
    \end{enumerate-abc}
\end{theorem}

Example~\ref{eg:drinker's paradox} and the derivation in Figure~\ref{fig:strict derivation} provide an instance of this theorem.

\begin{proof}
    We use the same argument as in the proof of Theorem~\ref{thm:approximate strong completeness of GŁ∀}. 
    The only difference is that when we construct $\DT_n(T, \calG)$, if we encounter a formula of the form $B\to C$, we will do nothing and move to the next stage.

    We claim that the construction must halt.
    If not, we argue same as in the proof of Lemma~\ref{lem:infinite tree implies invalid}, except that the case where $\calS$ is of the form $\Gamma\tto B\to C, \Delta$ or $\Gamma, B\to C\tto \Delta$ is omitted in the induction.
    The induction still goes through since both $T$ and $\calG$ are prenex.
    Hence there is a $[0, 1]$-model $\frakI$ of $T$ such that $\fs{\calG}_{\frakI}\geq\frac1n$, which contradicts to $T\models_{[0, 1]}\calG$.

    Therefore, the construction of $\DT_n(T, \calG)$ halts, then we follow the Lemma~\ref{lem:finite tree implies provable} to get a derivation of $m_1C_1\tto_{\frac1n}\bot\mid\cdots \mid m_kC_k\tto_{\frac1n}\bot \mid \calG_{\frac1n}$, where $m_1, \cdots, m_k\in\nn_{>0}$ and $C_1, \cdots, C_k\in T$.
    Note that $\DT_n(T, \calG)$ does not branch, hence $\DT_n(T, \calG)$ has only one leaf node $\sigma$.
    Then $\calG^n=\prop(\calG^{\sigma})_{\frac1n}$ is a desired mid-hypersequent by the proof of Lemma~\ref{lem:finite tree implies provable}.
\end{proof}

\begin{remark}\label{rmk:midhypersequent}
    If we restrict ourselves to the rules of $\GL\forall$, we cannot expect a single mid-hypersequent; instead, we can obtain many mid-hypersequents.
    This is due to the fact that, when eliminating the rule ($\tto_{\frac1n}\forall$) (and the dual case ($\exists\tto_{\frac1n}$)), several substituted instances of its premises have to be introduced; see the proof of Lemma~\ref{lem:1/n logical rules}.
\end{remark}

In the following, we prove a variant of approximate Herbrand's theorem of first-order Łukasiewicz logic, cf. Theorem~\ref{thm:compactness of approximate logical entailment}.

\begin{theorem}[Approximate Herbrand's theorem]\label{thm:approximate Herbrand's theorem}
    Let $\calL$ be a countable language, $T$ be a universal $\calL$-theory and $\exists\vec{x}A(\vec{x})$ be an $\calL$-formula, where $A(\vec{x})$ is quantifier-free.
    The following statements are equivalent:
    \begin{enumerate-abc}
        \item $T\models_{[0, 1]}\exists\vec{x}A(\vec{x})$.
        \item For each $n\in\nn_{>0}$, there are $\calL$-terms $\vec{t}_1, \cdots, \vec{t}_s$, a finite subset $T_n\subseteq T$ and $m\in\nn_{>0}$ such that
        \begin{equation*}
            \min_{1\leq i\leq s}\fs{A({\vec{t}_i})}_{\frakI}<\frac1n
        \end{equation*}
        for each $[0, 1]$-interpretation $\frakI$ satisfying $\max_{C\in T_n}\fs{C}_{\frakI}\leq\frac{1}{m}$.
        \item For each $n\in\nn_{>0}$, there are $\calL$-terms $\vec{t}_1, \cdots, \vec{t}_s$ and a finite subset $T_n\subseteq T$ such that
        \begin{equation*}
            \min_{1\leq i\leq s}\fs{A({\vec{t}_i})}_{\frakI}<\frac1n
        \end{equation*}
        for each $[0, 1]$-model $\frakI$ of $T_n$.
        \item For each $n\in\nn_{>0}$, there are $\calL$-terms $\vec{t}_1, \cdots, \vec{t}_s$ such that
        \begin{equation*}
            \min_{1\leq i\leq s}\fs{A({\vec{t}_i})}_{\frakI}<\frac1n
        \end{equation*}
        for each $[0, 1]$-model $\frakI$ of $T$.
    \end{enumerate-abc}
\end{theorem}

\begin{proof}
    (a) $\tto$ (b): 
    We fix $n\in\nn_{>0}$.
    By Theorem~\ref{thm:midhypersequent}, there are $m_1, \cdots, m_k\in\nn_{>0}$, $C_1,\cdots,C_k\in T$ and a propositional $\calL$-hypersequent $\calG^{n+1}$ such that $\calG^{n+1}$ is derivable in $\GL$ and 
    \begin{equation*}
        m_1 C_1 \tto_{\frac{1}{n+1}}\bot\mid\cdots \mid m_k C_k\tto_{\frac{1}{n+1}}\bot \mid {}\tto_{\frac{1}{n+1}} \exists\vec{x}A(\vec{x}) 
    \end{equation*}
    is derivable from $\calG^{n+1}$ by using ($\forall\tto_{\frac1{n+1}}$), ($\tto_{\frac1{n+1}}\forall$), ($\exists\tto_{\frac1{n+1}}$), ($\tto_{\frac1{n+1}}\exists$), (EW) and (EC).
    Hence $\calG^{n+1}$ must be of the form
    \begin{align*}
        {} &m_1D_1(\vec{t}_{11})\tto_{\frac{1}{n+1}}\bot\mid\cdots\mid m_1D_1(\vec{t}_{1s_1})\tto_{\frac{1}{n+1}}\bot\mid\cdots \mid\\
        {} & m_kD_k(\vec{t}_{k1})\tto_{\frac{1}{n+1}}\bot\mid\cdots\mid m_kD_k(\vec{t}_{ks_k})\tto_{\frac{1}{n+1}}\bot\mid\\
        {} & {}\tto_{\frac{1}{n+1}} A(\vec{t}_1) \mid\cdots\mid {}\tto_{\frac{1}{n+1}} A(\vec{t}_s), 
    \end{align*}
    where $C_i=\forall \vec{y} D_i(\vec{y})$, $D_i(\vec{y})$ is quantifier-free, $\vec{t}_{ij}$ and $\vec{t}_i$ are $\calL$-terms.
    Now we show that $\calL$-terms $\vec{t}_1, \cdots, \vec{t}_s$, finite subset $T_n=\{C_1,\cdots, C_k\}$ and positive integer $m=1+2\max_{1\leq i\leq k}{m_i}$ are the desired objects.

    Let $\frakI=(\calM, \beta)$ be a $[0, 1]$-interpretation satisfying $\max_{1\leq i\leq k}\fs{C_i}_{\frakI}\leq \frac{1}{m}$.
    Then for each $1\leq i\leq k$ and $1\leq j\leq s_i$, we have
    \begin{align*}
        \fs{D_i(\vec{t}_{ij})}_{\frakI}
        &=\fs{D_i(\vec{y})}_{\frakI[\vec{y}\mapsto \beta(\vec{t}_{ij})]}\\
        &\leq\sup_{\vec{m}\in M}\fs{D_i(\vec{y})}_{\frakI[\vec{y}\mapsto \vec{m}]}
        =\fs{\forall\vec{y}D_i(\vec{y})}_\frakI=\fs{C_i}_\frakI\leq \frac{1}{m}
    \end{align*}
    by substitution lemma (Proposition~\ref{prop:substitution lemma}).
    Using an argument similar to the proof of Theorem~\ref{thm:compactness of approximate logical entailment}, we can show that $\fs{m_iD_i(\vec{t}_{ij})\tto_{\frac{1}{n+1}}\bot}_{\frakI}>0$ for each $1\leq i\leq k$ and $1\leq j\leq s_i$, hence there must be some $1\leq i_0\leq s$ such that $\fs{{}\tto_{\frac{1}{n+1}} A(\vec{t}_{i_0})}_{\frakI}\leq 0$, i.e., $\fs{A(\vec{t}_{i_0})}_{\frakI}\leq \frac{1}{n+1}$.
    Then we have
    \begin{equation*}
        \min_{1\leq i\leq s} \fs{A({\vec{t}_i})}_{\frakI}
        \leq \fs{A(\vec{t}_{i_0})}_{\frakI}
        \leq \frac{1}{n+1}<\frac{1}{n}.
    \end{equation*}

    (b) $\tto$ (c) $\tto$ (d): In a way similar to the proof of Theorem~\ref{thm:compactness of approximate logical entailment}.

    (d) $\tto$ (a): Let $\frakI=(\calM, \beta)$ be a $[0, 1]$-model of $T$.
    For each $n\in\nn_{>0}$, there are $\calL$-terms $\vec{t}_1, \cdots, \vec{t}_s$ such that $\min_{1\leq i\leq s} \fs{A({\vec{t}_i})}_{\frakI}<\frac{1}{n}$.
    Hence 
    \begin{align*}
        \fs{\exists \vec{x}A(\vec{x})}_{\frakI}
        &=\inf_{\vec{m}\in M} \fs{A(\vec{x})}_{\frakI[\vec{x}\mapsto \vec{m}]}
        \leq \min_{1\leq i\leq s} \fs{A(\vec{x})}_{\frakI[\vec{x}\mapsto \beta(\vec{t}_{i})]}\\
        &= \min_{1\leq i\leq s} \fs{A(\vec{t}_{i})}_{\frakI}<\frac{1}{n}
    \end{align*}
    by substitution lemma (Proposition~\ref{prop:substitution lemma}).
    Since $n$ is arbitrary, it follows that $\fs{\exists\vec{x}A(\vec{x})}_{\frakI}=0$.
\end{proof}

\begin{remark}~
    \begin{enumerate}
        \item The implication (a) $\tto$ (b) can also be shown semantically by a compactness argument; see \cite[Theorem 3]{baaz2010} for a proof in a special case.
        \item The equivalence (a) $\otto$ (d) has already been established under a more general setting in \cite{cintula2019}, and a version for continuous logic can be found in \cite{goldbring2012}.
    \end{enumerate}
\end{remark}

We conclude this paper by discussing some potential applications of our results.

Theorem~\ref{thm:approximate strong completeness of GŁ∀+(s-Cut)} implies the admissibility of free cuts; see \cite[Chapter 1, 2.4.4]{buss1998} for free cuts.
One may expect a free-cut elimination theorem for $\GL\forall+(\text{s-Cut})$ analogous to \cite[Chapter 1, 2.4.5]{buss1998}.

It is also possible to formalize the approximate strong completeness of continuous logic in $\GL\forall$.
Continuous logic is extended from first-order Łukasiewicz logic by adding a unary connective $\frac12$, a binary predicate symbol $d$, axioms of $\frac12$ and $d$, and axioms about uniform continuity for all function and predicate symbols.
Ben Yaacov and Pedersen~\cite{benyaacov2010} showed the approximate strong completeness of this logic, which suggests that a similar result may also be obtained for $\GL\forall$ (or its extensions).
The idea has already appeared in \cite{wei2025}, but our results provide further support for this approach.

\bibliographystyle{plain}
\bibliography{ref.bib}

\end{document}